\documentclass[preprint]{imsart}
\newcommand{\diag}{\mbox{diag}}
\newcommand{\E}{\mathbb{E}}
\newcommand{\Prob}{\mathbb{P}}

\usepackage{etex}
\usepackage[figuresright]{rotating}
\usepackage{amsmath}
\usepackage{amssymb}
\usepackage{bbm}
\usepackage{subfigure}
\usepackage{pstricks,pst-node,pst-plot}
\usepackage{color}
\usepackage{pst-all}
\usepackage{pspicture}
\usepackage{booktabs}
\usepackage{wasysym}
\usepackage[all]{xy}
\usepackage{epsfig,graphicx}
\usepackage{theorem}
\usepackage{natbib}
\usepackage{tikz}
\usepackage{algorithm}
\usepackage{algorithmic}
\usepackage{multirow}

\newcounter{count}
\newcounter{asscount}

\newenvironment{theorem}[1][Theorem \arabic{count}]{\vspace{1em}\refstepcounter{count}\begin{trivlist}
\item[\hskip \labelsep {\bfseries #1}]\em}{\end{trivlist}\vspace{1em}}

\newenvironment{lemma}[1][Lemma \arabic{count}]{\vspace{1em}\refstepcounter{count}\begin{trivlist}
\item[\hskip \labelsep {\bfseries #1}]\em}{\end{trivlist}\vspace{1em}}

\newenvironment{corollary}[1][Corollary \arabic{count}]{\vspace{1em}\refstepcounter{count}\begin{trivlist}
\item[\hskip \labelsep {\bfseries #1}]\em}{\end{trivlist}\vspace{1em}}

\newenvironment{proposition}[1][Proposition \arabic{count}]{\vspace{1em}\refstepcounter{count}\begin{trivlist}
\item[\hskip \labelsep {\bfseries #1}]\em}{\end{trivlist}\vspace{1em}}

\newenvironment{proof}[1][Proof]{\vspace{1em}\begin{trivlist}
\item[\hskip \labelsep {\bfseries #1}]}{\hfill$\Box$\end{trivlist}\vspace{1em}}

\newenvironment{assumption}[1][Assumption \arabic{asscount}]{\vspace{1em}\refstepcounter{asscount}\begin{trivlist}
\item[\hskip \labelsep {\bfseries #1}]}{\end{trivlist}\vspace{1em}}

\begin{document}
\begin{frontmatter}

\title{Logistic regression and Ising networks: prediction and estimation when violating lasso assumptions}
\runtitle{}
\author{Lourens Waldorp}
\author{Maarten Marsman}
\author{Gunter Maris}
\runauthor{Waldorp et al.}
\runtitle{Ising model and logistic regression}
\address{University of Amsterdam, Nieuwe Achtergracht 129-B, 1018 WB, the Netherlands}

\begin{abstract}
The Ising model was originally developed to model magnetisation of solids in statistical physics. As a network of binary variables with the probability of becoming 'active' depending only on direct neighbours, the Ising model appears appropriate for many other processes. For instance, it was recently applied in psychology to model co-occurrences of mental disorders. 
It has been shown that the connections between the variables (nodes) in the Ising network can be estimated with a series of logistic regressions. This naturally leads to questions of how well such a model predicts new observations and how well parameters of the Ising model can be estimated using logistic regressions. Here we focus on the high-dimensional setting with more parameters than observations and consider violations of assumptions of the lasso. In particular, we determine the consequences for both prediction and estimation when the sparsity and restricted eigenvalue assumptions are not satisfied. We explain by using the idea of connected copies (extreme multicollinearity) the fact that prediction becomes better when either sparsity or multicollinearity is not satisfied. We illustrate these results with simulations. 
\end{abstract}


\end{frontmatter}

\section{Introduction}\noindent
Logistic regression models the ratio of success versus failure for binary variables. These models are convenient and useful in many situations. To name one example, in genome wide association scans (GWAS) a combination of alleles on single nucleotide polymorphisms (SNPs) is either present or not \citep{Cantor:2010}. Recently it became clear that logistic regression can also be used to obtain estimates of connections in a binary network \citep[e.g., ][]{Ravikumar:2010,Buhlmann:2011}. A particular version of a binary network is the Ising model, in which the probability of a node being 'active' is determined only by its direct neighbours (pairwise interactions only). The Ising model originated in statistical physics and was used to model magnetisation of solids \citep{Kindermann:1980,Cipra:1987,Baxter:2007}, and was investigated extensively by \citet{Besag:1974} and \citet{Cressie:1993} and recently by \citet{Marsman:2017}, amongst others, in a statistical modelling context. Recently, the Ising model has also been applied to modelling networks of mental disorders \citep{Borsboom:2011,Borkulo:2014}. The objective in models of psychopathology are to both explain and predict certain observations like co-occurrences of disorders (comorbidity).  

Here we focus on violations of the assumptions of lasso in logistic regression with high-dimensional data (more parameters than observations, $p>n$). In particular, we consider the consequences for both prediction and estimation when violating the assumptions of sparsity and restricted eigenvalues (multicollinearity). For sparse models and $p>n$ it has been shown that statistical guarantees about the underlying network and its coefficients can be obtained with certain assumptions for Gaussian data \citep{Meinshausen:2006,Bickel:2009,Hastie:2015}, for discrete data \citep{Loh:2012}, and for exponental family distributions \citep{Buhlmann:2011}. Specifically, \citet{Ravikumar:2010} show that, under strong regularity conditions, using a series of regressions for the conditional probability in the Ising model (logistic regression), the correct structure (topology) of a network can be obtained in the high-dimensional setting. 

In many practical settings it is uncertain whether the assumptions of the lasso for accurate network estimation hold. Specifically, the assumptions of sparsity and the restricted eigenvalues \citep{Bickel:2009} are in many situations untestable. We therefore investigate here how estimation and prediction in Ising networks are affected by violating the sparsity and restricted eigenvalue assumptions. The setting of logistic regression and nodewise estimation of the Ising model parameters allows us to clearly determine how and why prediction and estimation are affected. 
We use the idea of connected nodes in a graph that are identical in the observations (and call them connected copies) to show why prediction is better for graph structures that violate the restricted eigenvalue or sparsity assumption. These connected copies represent the idea of extreme multicollinearity. One way to view connected copies is obtaining edge weights that lead to a network with perfect correlations between nodes (variables). We therefore compare in terms of prediction and estimation different situations where we violate the restricted eigenvalue or sparsity assumption based on different data generating processes. An example of a setting where near connected copies in networks are found is in high resolution functional magnetic resonance imaging. Here time series of contiguous voxels, that are connected \citep[also physically, see e.g., ][]{Johansen-Berg:2004}, are near exact copies of one another. The concept of connected copies allows us to determine the consequences for prediction loss, by using the fact that subsets of connected copies do not change the risk or $\ell_{1}$ norm. We also show that prediction loss is a lower bound for estimation error (in $\ell_{1}$) and so by consequence, if prediction loss increases, so does estimation error. 

We first provide some background in Section \ref{sec:log-regression} on the Ising model and its relation to logistic regression. To show the consequences of violating the assumptions of multicollinearity and sparsity we discuss these assumptions at length in Section \ref{sec:assumptions}. We also show how they provide the statistical guarantees for the lasso \citep[e.g., ][]{Wainright:2012a,Buhlmann:2011,Ravikumar:2010}. Then armed with these intuitions, we give in Section \ref{sec:violations} some insight into the consequences for prediction and estimation when the sparsity or restricted eigenvalue assumption is violated. We also provide some simulations to confirm our results. 

\section{Logistic regression and the Ising model}\label{sec:log-regression}
The Ising model is part of the exponential family of distributions \citep[see, e.g., ][]{Brown:1986,Young:2005,Wainwright:2008}. Let $X=(X_{1},X_{2},\ldots,X_{p})$ be a random variable with values in $\{0,1\}^{p}$. 
The Ising model can then be defined as follows. Let $G$ be a graph consisting of nodes in $V=\{1,2,\ldots,p\}$ and edges $(s,t)$ in $E\subseteq V\times V$. To each node $s\in V$ a random variable $X_{s}$ is associated with values in  $\{0,1\}$. The probability of each configuration $x$ depends on a main effect (external field) and pairwise interactions. It is sometimes referred to as the auto logistic-function \citep{Besag:1974}, or a pairwise Markov random field, to emphasise that the parameter and sufficient statistic space are limited to pairwise interactions \citep{Wainwright:2008}. Each $x_{s}\in \{0,1\}$ has conditional on all remaining variables (nodes) $X_{\backslash s}$ probability of success $\pi_{s}:=\Prob(X_{s}=1\mid x_{\backslash s})$, where $x_{\backslash s}$ contains all nodes except $s$.
Let $\xi=(m,A)$ contain all parameters, where the $p\times p$ matrix $A$ contains the pairwise interaction strengths and the $p$ vector $m$ is the main effects (external field). The distribution for configuration $x$ of the Ising model is then 
\begin{align}
\Prob(x) = \frac{1}{Z(\xi)}\exp\left( \sum_{s\in V}m_{s}x_{s}+\sum_{(s,t)\in E}A_{st}x_{s}x_{t}\right)
\end{align}
In general, the normalisation $Z(\xi)$ is intractable, because the sum consists of $2^{p}$ possible configurations for $x\in\{0,1\}^{p}$; for example, for $p=30$ we obtain over 1 million configurations to evaluate in the sum in $Z(\xi)$ (see, e.g., \citet{Wainwright:2008} for lattice (Bethe) approximations).

Alternatively, the conditional distribution does not contain the normalisation constant $Z(\xi)$ and so is more amenable to analysis. The conditional distribution is again an Ising model \citep{Besag:1974,Kolaczyk:2009}
\begin{align}\label{eq:cond-prob}
\pi_{s}=\Prob(X_{s}=1\mid x_{\backslash s}) = \frac{\exp\left( m_{s}+\sum_{t:(s,t)\in E}A_{st}x_{t} \right)}{1+\exp\left( m_{s}+\sum_{t:(s,t)\in E}A_{st}x_{t} \right)}.
\end{align}
It immediately follows that the log-odds \citep{Besag:1974} is 
\begin{align}\label{eq:log-odds}
\mu_{s}(x_{\backslash s})=\log\left(\frac{\pi_{s}}{1-\pi_{s}}\right)
=m_{s} +\sum_{t:(s,t)\in E}A_{st}x_{t}.
\end{align}
For each node $s\in V$ we collect the $p$ parameters $m_{s}$ and $(A_{st}, t\in V\backslash \{s\})$ in the vector $\theta$. Note that the log-odds $\theta\mapsto\mu_{\theta}$ is a linear function, and so if $x=(1,x_{\backslash s})$ then $\mu_{\theta}=x^{\sf T}\theta$. The theory of generalised linear models (GLM) can therefore immediately be applied to yield consistent and efficient estimators of $\theta$ when sufficient observations are available, i.e., $p<n$ \citep{Nelder:1972,Demidenko:2004}. To obtain an estimate of $\theta$ when $p>n$, we require regularisation or another method \citep{Hastie:2015,Buhlmann:2013}. 

\subsection{Nodewise logistic regression}\noindent
\citet{Meinshausen:2006} showed that for sparse models the true neighbourhood of a graph can be obtained with high probability by performing a series of conditional regressions with Gaussian random variables. For each node $s\in V$ the set of nodes with nonzero $A_{st}$ are determined, culminating in a neighbourhood for each node. Combining these results leads to the complete graph, even when the number of nodes $p$ is much larger than the number of observations $n$. This is called neighbourhood selection, or nodewise regression. This idea was extended to Bernoulli (Ising) graphs by \citet{Ravikumar:2010}, but see also \citet[][chapters 6 and 13]{Buhlmann:2011}.
Nodewise regression allows us to use standard logistic regression to determine the neighbourhood for each node. This framework, of course, comes at a cost, and two strong assumptions are required. We discuss these assumptions in Section \ref{sec:assumptions}. 

To estimate the coefficients, \citet{Meinshausen:2006} used a sequential regression procedure for Gaussian data where each node in turn is treated as the dependent variable and the remaining ones as independent variables. By repeating this analysis for all nodes in $V$, a total of $p-1$ neighbourhood estimates of nonzero parameters are obtained for all nodes $s\in V$. Since each node is considered twice, the estimates are often combined by either an {\em and}-rule, where an edge is obtained if $\hat{A}_{st}\ne 0$ and $\hat{A}_{ts}\ne 0$, or an {\em or}-rule, where either parameter estimate can be nonzero \citep{Meinshausen:2006}. 

\citet{Ravikumar:2010} translated this procedure to binary variables using pseudo-likelihoods. Recall that $\theta\mapsto \mu_{\theta}$ is the linear function $\mu_{\theta_{s}}(x_{\backslash s})= m_{s}+\sum_{t\in V\backslash s}A_{st}x_{t}$ of the conditional Ising model obtained from the log-odds (\ref{eq:log-odds}). The parameters in the $p$ dimensional vector $\theta$ are $m_{s}$ for the intercept (external field) and $(A_{st},t\in V\backslash \{s\}))$, representing the connectivity parameters for node $s$ based on all remaining nodes $V\backslash \{s\}$. Let the $n\times p$ matrix $X_{\backslash s}=(1_{n},X_{1}, \ldots,X_{p})$ be the matrix with the vector of 1s in $1_{n}$ and the remaining variables without $X_{s}$. We write  $y_{i}$ for the observation $x_{i,s}$ of node $s$ and $x_{i}=(1,x_{i,\backslash s})$ and  $\mu_{i}:=\mu_{i,\theta_{s}}(x_{i,\backslash s})$, basically leaving out the subscript $s$ to index the node, and only use the node index $s$ whenever circumstances demand it. Let the loss function be the negative log of the conditional probability $\pi$ in (\ref{eq:cond-prob}), known as a pseudo log-likelihood \citep{Besag:1974}
\begin{align}\label{eq:psi-loss}
\psi(y_{i},\mu_{i}) :=-\log \Prob (y_{i}\mid x_{i}) = -y_{i}\mu_{i} + \log(1+ \exp(\mu_{i})).
\end{align}
For logistic loss $\psi$ the theoretical risk is defined as 
\begin{align}\label{eq:risk}
R_{\psi}(\mu) =  \frac{1}{n}\sum_{i=1}^{n}\E\psi(y_{i},\mu_{i}).
\end{align}
The value that optimises the risk is $\theta^{*}=\arg\inf_{\theta\in\mathbb{R}^{p}} R_{\psi}(\mu)$; given the choice of logistic loss we can do no better than $\theta^{*}$ in terms of the population. Of course we do not have the theoretical risk and so we use an empirical version
\begin{align}\label{eq:emp-risk}
R_{n,\psi}(\mu)=\frac{1}{n}\sum_{i=1}^{n}\psi(y_{i},\mu_{i}).
\end{align}
Define $\mu^{*}:=\mu_{\theta^{*}}(x)$, which uses the optimal value $\theta^{*}$ under theoretical risk. For sparse estimation the $\ell_{1}$ (lasso) minimisation is given by 
\begin{align}\label{eq:lasso}
\hat{\theta}=\arg \min_{\theta\in\mathbb{R}^{p}} \left\{\frac{1}{n}\sum_{i=1}^{n}\psi(y_{i},\mu_{i})+\lambda||\theta||_{1} \right\}
\end{align}
where $||\theta||_{1}=\sum_{t\in V\backslash\{s\}}|\theta_{t}|$ is the $\ell_{1}$ norm, $\lambda$ is a fixed penalty parameter. Since $\psi$ is convex and $||\theta||_{1}$ is convex, the objective function $R_{n,\psi}+\lambda||\theta||_{1}$ in (\ref{eq:lasso}) is convex, which allows us to apply convex optimisation. We discuss how to obtain the parameters with the coordinate descent algorithm in Section \ref{sec:violations}.

Once the parameters are obtained it turns out that inference on network parameters is in general difficult with $\ell_{1}$ regularisation \citep{Potscher:2009b}. One solution is to desparsify it by adding a projection of the residuals \citep{Geer2013,Javanmard:2014,Zhang:2014,Waldorp:2015b}, which is sometimes referred to as the desparsified lasso. Another type of inference is one where clusters of nodes obtained from the lasso are interpreted instead of individual nodes \citep{Lockhart:2014}.

To illustrate the result of an implementation of logistic regression for the Ising model, consider Figure \ref{fig:ising-graphs}. We generated a random Erd\"{o}s-Renyi graph (left panel) with $p=100$ nodes and probability of an edge 0.05, resulting in 258 edges.  The {\em igraph} package in {\small\sf R} was used with {\em erdos.renyi.game} \citep{Csardi:2006}. To generate data ($n=50$ observations of the $p=100$ nodes) from the Ising model the package {\em IsingSampler} was used, and to obtain estimates of the parameters the package {\em IsingFit} was used \citep{Borkulo:2014} in combination with the {\em and} rule.
\begin{figure}[t!]
\centering
\includegraphics[width=5.5cm]{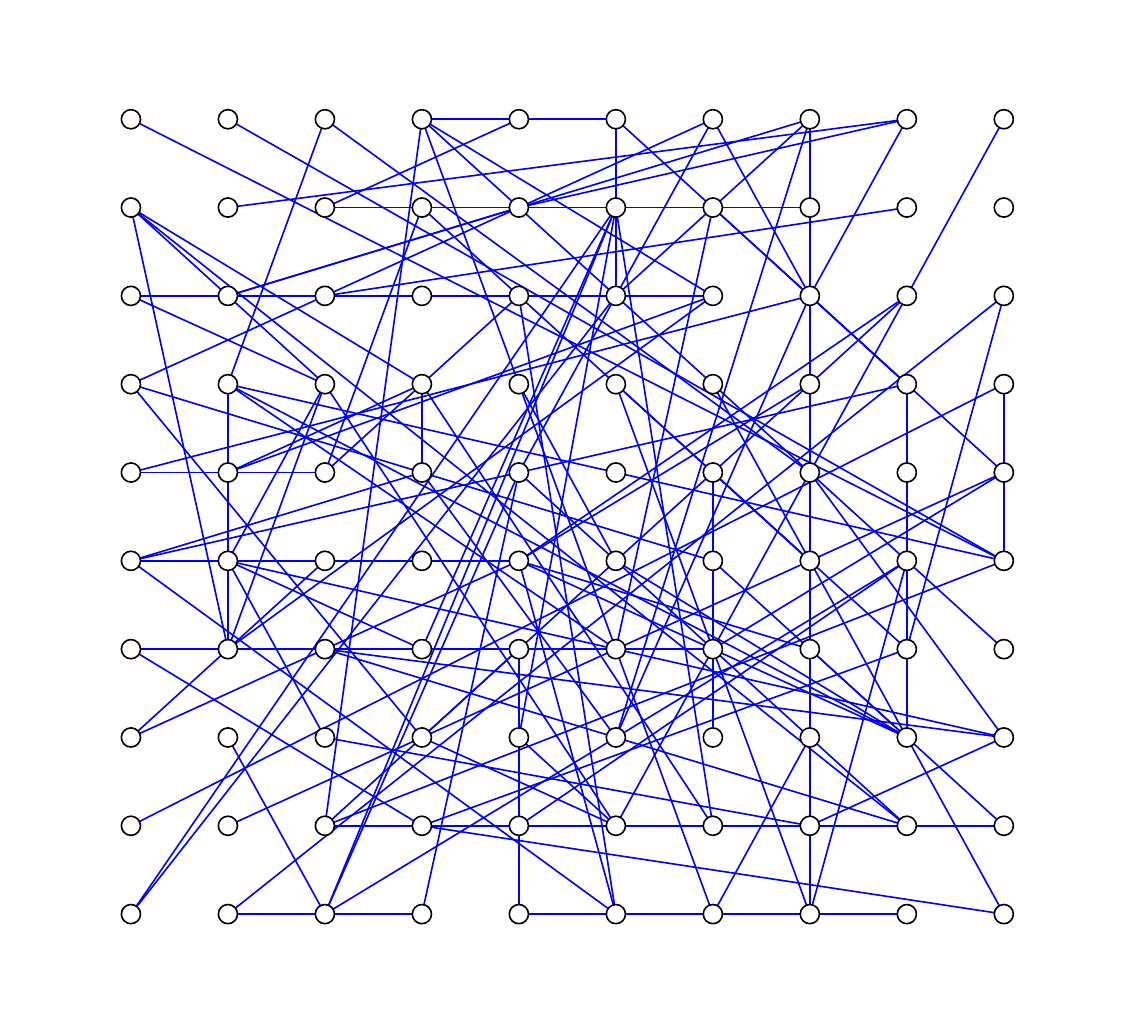}
~
\includegraphics[width=5.5cm]{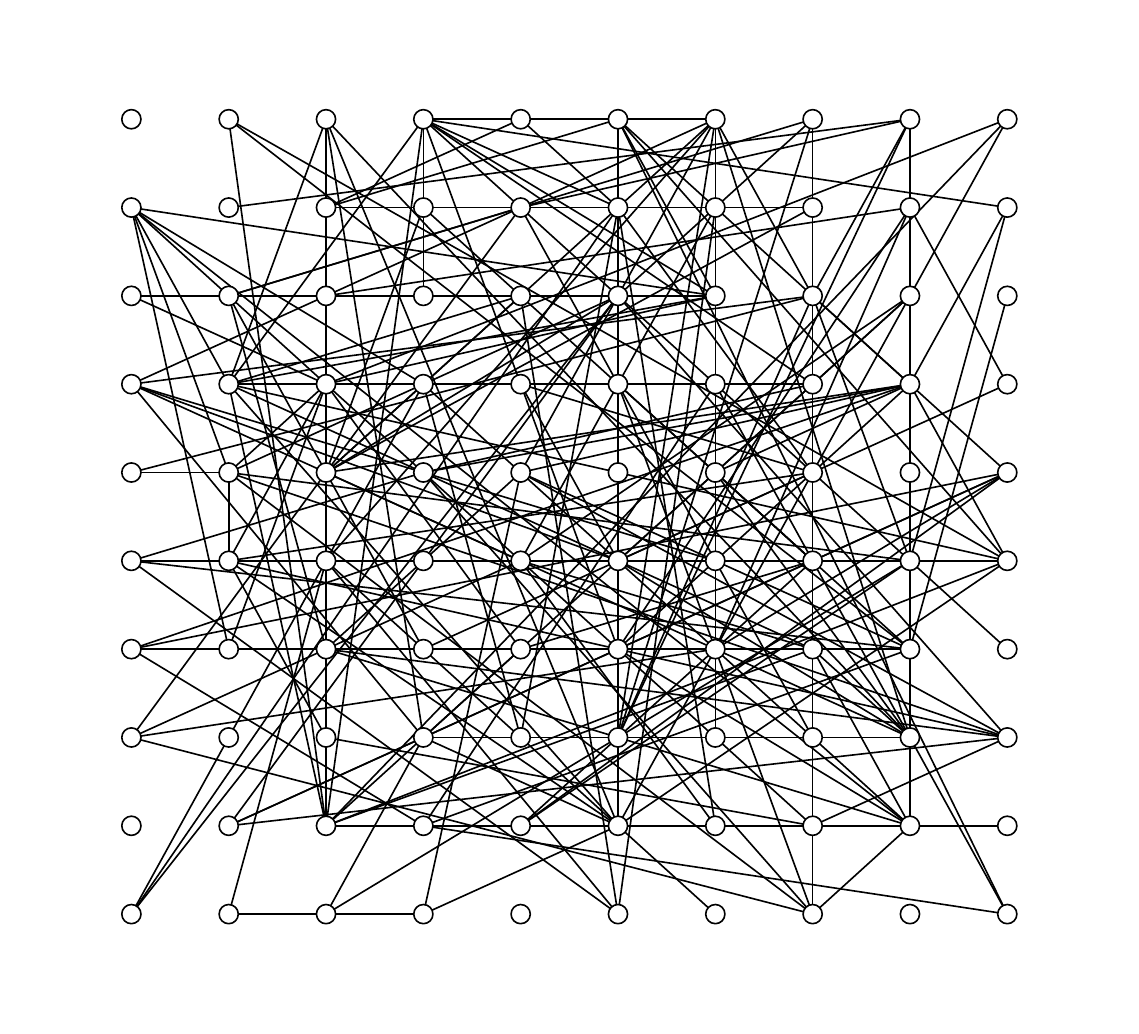}
\caption{Ising networks with $p=100$ nodes. Left panel: true network used to generate data. Right panel: estimated Ising model with nodewise logistic regression from $n=50$ generated observations.  }
\label{fig:ising-graphs}
\end{figure}
The recall (true positive rate) for this example was 0.69 and the precision (positive predictive value) was 0.42. So we see that about 30\% of the true edges are missing and about 60\% of the estimated edges is incorrect. This is not surprising given that we have 4950 possible edges to determine and only 50 observations. (More details on the simulation are in Section \ref{sec:numerical}.)

\section{Assumptions for prediction and estimation}\label{sec:assumptions}
In order to determine the consequences of violating the assumptions of the lasso in logistic regression, we discuss the assumptions for accurate prediction and estimation. Both prediction and estimation require that the solution is sparse; informally, that the number of non-zero edges in the graph is relatively small (see Assumption \ref{ass:sparsity} below). For accurate estimation we also require an assumption on the covariance between the nodes in the graph. Several types of assumptions have been proposed \citep[see ][ for an excellent overview and additional results on obtaining the lasso solution]{Geer:2009}, but here we focus on the restricted eigenvalue assumption because of its direct connection to multicollinearity. 

\subsection{Sparsity}\label{sec:sparsity}
Central to lasso estimation is the assumption that the underlying problem is low-dimensional \citep{Buhlmann:2011,Giraud:2014}. This is the assumption of sparsity. This is essential because whenever $p>n$ there is no unique solution to  the empirical risk $R_{n,\psi}(\mu)$ defined in (\ref{eq:emp-risk}) \citep{Wainwright:2009}. Sparsity can be defined in different ways. The most common is a restriction on the number of nonzero edges, sometimes referred to as coordinate sparsity \citep{Giraud:2014}. Let $S_{0}$ denote the support containing the indices of the nonzero coefficients, i.e., $S_{0}:=\{j: \theta_{j}\ne 0\}$ and its size $s_{0}=|S_{0}|$. %
\begin{assumption}\label{ass:sparsity}
(Coordinate sparsity) The size $s_{0}$ of the set of nonzero coefficients $S_{0}$ in $\theta^{*}$ is of order $o(\sqrt{n/\log p})$. 
\end{assumption}
There are other forms of sparsity, like the fused sparsity, where the support is defined as $\{j:\theta_{j}-\theta_{j-1}\ne 0\}$. This ensures that there are relatively few jumps in, for instance, a piecewise continuous function \citep[see][for more details]{Giraud:2014}. 
Another form of sparsity is where the $\ell_{1}$ size of the parameter vector $\theta$ is restricted. We use this to show that prediction (classification) in logistic regression is accurate.
\begin{assumption}\label{ass:sparsity-ell1}
($\ell_{1}$-sparsity) The $\ell_{1}$ norm of the coefficients $\theta^{*}$ is of order $o(\sqrt{n/\log p})$, i.e., $||\theta^{*}||_{1}=o(\sqrt{n/\log p})$. 
\end{assumption}
In logistic regression there is a natural classifier that predicts whether $y_{i}$ is 1 or 0. We simply check whether the probability of a 1 is greater than 1/2, that is, whether $\pi_{i} > 1/2$. Because $\mu_{i}>0$ if and only if $\pi_{i}>1/2$ we obtain the natural classifier
\begin{align}\label{eq:classifier}
\mathcal{C}(y_{i}) = \mathbbm{1}\{ \mu_{i}>0 \}
\end{align}
where $\mathbbm{1}$ is the indicator function. This is called 0-1 loss or sometimes Bayes loss \citep{Hastie:2015}. 
Instead of 0-1 loss we use logistic loss (\ref{eq:psi-loss}) to determine how well we predict individual observations $y_{i}$ to which class they belong, 0 or 1. Define the prediction loss (sometimes called excess risk) with logistic loss $\psi$ as
\begin{align}\label{eq:prediction-loss}
\mathcal{L}_{\psi}(\mu) = \frac{1}{n}\sum_{i=1}^{n}\E(\psi(y_{i},\mu_{i})-\psi(y_{i},\mu_{i}^{*})).
\end{align}
Note that by definition of $\theta^{*}$ that $\mathcal{L}_{\psi}(\mu)\ge 0$ for any $\theta\mapsto\mu_{\theta}$. A similar definition is possible for 0-1 loss using $\mathcal{C}$, which is $\mathcal{L}_{\mathcal{C}}(\mu)$. 
\citet[][Theorem 3.3]{Bartlett:2003} show that $\mathcal{L}_{\psi}\to 0$ implies $\mathcal{L}_{\mathcal{C}}\to 0$ as $n\to \infty$. In other words, using logistic loss eventually results in the optimal 0-1 (Bayes) loss, and so nothing is lost in using logistic loss as a proxy for 0-1 loss.

Prediction has been shown to be accurate using Assumption \ref{ass:sparsity-ell1}. Suppose that the regularisation parameter $\lambda$ is of order $O(\sqrt{\log p/n})$, then the prediction loss is bounded above \citep{Buhlmann:2011}
\begin{align}\label{eq:pred-loss-const}
\mathcal{L}_{\psi}(\hat{\mu}) + \lambda||\hat{\theta}||_{1}
\le 2\lambda ||\theta^{*}||_{1} 
\end{align}
If in addition Assumption \ref{ass:sparsity-ell1} holds, where $||\theta^{*}||_{1}$ is of order $o(\sqrt{n/\log p})$, then $\mathcal{L}_{\psi}(\hat{\mu})=o(1)$. This result is in the Appendix as Lemma \ref{lem:pred-loss-const} and corresponds to that in \citet{Ravikumar:2010}, but see also \citet[][Section 14.8]{Buhlmann:2011} for stronger results. The requirement that the regularisation parameter is of order $O(\sqrt{\log p/n})$ is obtained because the stochastic part in the prediction loss has to be negligible (see Lemma \ref{lem:emp-proc-bound} in the Appendix for details). If we choose $\lambda$ sufficiently large, we are guaranteed with probability at least $1-2\exp(-nt^{2})$ for some $t>0$ that the prediction loss is bounded by the $\ell_{1}$ norm of the parameter of interest $\theta^{*}$ as in (\ref{eq:pred-loss-const}). 

It follows directly from (\ref{eq:pred-loss-const}) that the lasso estimation error is larger than prediction loss, and so prediction is easier than estimation \citep[see also][]{Hastie:2001}. From (\ref{eq:pred-loss-const}) we get an upper bound on prediction loss
\begin{align}\label{eq:prediction-bound-estimation}
\mathcal{L}_{\psi}(\hat{\mu}) 
\le 2\lambda(||\hat{\theta}-\theta^{*}||_{1})
\end{align}
where we used the reverse triangle inequality (see Lemma \ref{lem:prediction-error-bound} in the Appendix for details). This shows that lasso estimation error is larger than prediction error.


\subsection{Restricted eigenvalues}\label{sec:restricted-eigenvalues}
Next to sparsity, the second assumption for the lasso is related to the problem that when $p>n$ the empirical risk $R_{n,\psi}$ is not strongly convex and hence no unique solution is available. It turns out that we need to consider a subset of lasso estimation errors $\delta=\hat{\theta}-\theta^{*}$ such that strong convexity holds for that subset \citep{Wainright:2012a}. 

Because we have $p>n$ we cannot obtain strong convexity in general, and we need to relax the assumption. This is how we get to the restricted eigenvalue assumption. Let $\nabla_{j}\psi(y_{i},x_{i}^{\sf T}\theta)$ be the first derivative with respect $\theta_{j}$ and $\nabla^{2}_{jj}\psi(y_{i},x_{i}^{\sf T}\theta)$ the second derivative with respect to $\theta_{j}$. Then demanding strong convexity means that, if we consider the $s_{0}\times s_{0}$  submatrix $\nabla^{2}_{S_{0}}\psi_{n}(\theta)$ then we need that $\nabla^{2}_{S_{0}}\psi_{n}(\theta)\ge \gamma I$, where $I$ is the identity matrix and we used $\psi(\theta)$ instead of $\psi(y,\mu)$ to emphasise dependence on $\theta$ (and $\mu=x^{\sf T}\theta$).  This we can never get (see the Appendix for more details on strong convexity).
But from (\ref{eq:pred-loss-const}) it follows that if the directions of the lasso error $\delta=\hat{\theta}-\theta^{*}$  follow a cone shaped region with $||\delta_{S_{0}^{c}}||_{1}\le \alpha||\delta_{S_{0}}||_{1}$ (see Theorem \ref{thm:est-err} in the Appendix), then within these directions strong convexity holds.  We refer to this set as $\mathbb{C}_{\alpha}=\{\delta\in\mathbb{R}^{p}:||\delta_{S_{0}^{c}}||_{1}\le \alpha||\delta_{S_{0}}||_{1}\}$. In the directions where the cone shape holds so that $\delta\in\mathbb{C}_{\alpha}$, the loss function is strictly larger than 0, except at $\delta=0$, but is flat and can be 0 if $\delta\notin\mathbb{C}_{\alpha}$ (see \citet{Wainright:2012a} or \citet{Hastie:2015} for an excellent discussion).  This assumption is called the restricted eigenvalue assumption. 

The second derivative or Fisher information matrix is 
\begin{align}\label{eq:second-deriv}
\nabla^{2}\psi(\theta)=\frac{1}{n}\sum_{i=1}^{n}\E \pi(\mu_{i})\pi(-\mu_{i})x_{i}x_{i}^{\sf T}.
\end{align}
We assume that this population level matrix is positive definite. Then by strong convexity we have for $\gamma>0$ that $\nabla^{2}\psi\ge \gamma I$, and so 
\begin{align*}
\mathcal{L}_{\psi}(\hat{\mu})\ge \frac{1}{2}\delta^{\sf T}\nabla^{2}\psi(\hat{\theta})\delta\ge\frac{\gamma}{2}||\delta||_{2}^{2}
\end{align*}
which allows us to relate the lasso estimation error to prediction loss such that  we can conclude consistency because of the bound on prediction error in (\ref{eq:pred-loss-const}) (see Lemma \ref{lem:pred-loss-const} in the Appendix). 
The problem is that we work with the empirical $p\times p$ matrix $\nabla^{2}\psi_{n}(\theta)$ which is necessarily singular since $p>n$. 
The empirical Fisher information is
\begin{align}\label{eq:second-deriv-emp}
\nabla^{2}\psi_{n}(\theta)=\frac{1}{n}\sum_{i=1}^{n}\pi(\mu_{i})\pi(-\mu_{i})x_{i}x_{i}^{\sf T}
\end{align}
which has zero eigenvalues because it is positive semidefinite whenever $p>n$. \citet{Bickel:2009} suggested the restricted eigenvalue assumption that is sufficient to guarantee that $\nabla^{2}\psi_{n}(\theta)$ has positive eigenvalues for lasso errors $\delta\in\mathbb{C}_{\alpha}$. Here we require in the setting of nodewise logistic regression that for all nodes $s$ simultaneously the lower bound $\gamma_{G}>0$ is sufficient and $\alpha=1$. We emphasise the nodewise estimation of all edges in $E$ by using $\psi_{s}$ and $\delta_{s}$.
\begin{assumption}\label{ass:RE}
(Restricted eigenvalue) 
The population Fisher information matrix $\nabla^{2}\psi_{s}$ of dimensions $p\times p$ is nonsingular and $\max_{j}\nabla^{2}_{jj}\psi_{s}(\theta)<K$, for some $K>0$ and for all $s\in V$. The empirical matrix $\nabla^{2}\psi_{n,s}(\theta)$ satisfies the restricted eigenvalue (RE) assumption if for some $\gamma_{G}>0$ it holds that 
\begin{align}
\min_{s\in V}\frac{\delta_{s}^{\sf T}\nabla^{2}\psi_{n,s}(\theta)\delta_{s}}{||\delta_{s}||_{2}^{2}} \ge\gamma_{G}\qquad \text{for all } 0\ne\delta_{s}\in\mathbb{C}_{1}. 
\end{align}
\end{assumption}
The restricted eigenvalue assumption has been investigated in the context of Gaussian data \citep[][chapter 11]{Bickel:2009,Wainwright:2009,Raskutti:2010,Hastie:2015}, in the setting of the Ising model \citep[][Lemma 3]{Ravikumar:2010}, and in generalised linear models \citep[][chapter 6]{Geer:2008,Buhlmann:2011}. The original restricted eigenvalue assumption as presented in \citet{Bickel:2009} is slightly stronger than the compatibility assumption of \citet{Geer:2009}.  See \citet{Geer:2009} for more details on the compatibility and other assumptions to bound estimation error in the lasso. Here we use the RE assumption because of its connection to multicollinearity, discussed in Section \ref{sec:violations}.

Let $\theta_{S}$ be the vector where for each $t\in V$ we have $\theta_{t}\mathbbm{1}\{t\in S\}$. It follows that $\theta=\theta_{S}+\theta_{S^{c}}$, where $S^{c}$ is the complement of $S$. 
The RE assumption implies that the $s_{0}\times s_{0}$ submatrix $\nabla^{2}_{S_{0}}\psi_{n}(\theta)$ indexed by $S_{0}$ has smallest eigenvalue $>0$. This can be seen as follows. RE implies that there is a $\delta$ such that $\delta_{S_{0}}\ne0$, $\delta_{S_{0}^{c}}=0$, implying $||\delta_{S_{0}^{c}}||_{1}\le ||\delta_{S_{0}}||_{1}$, and $\delta^{\sf T}(\nabla^{2}\psi_{n})\delta>0$. This implies that for some $\gamma_{G}>0$
\begin{align*}
\nabla^{2}_{S_{0}}\psi_{n}(\theta)\ge\gamma_{G} I
\end{align*}
and so we have restricted strong convexity for this $\delta\in \mathbb{C}_{1}$. The two Assumptions \ref{ass:sparsity} on coordinate sparsity and \ref{ass:RE} on restricted eigenvalues make it possible to derive the $\ell_{1}$ estimation error bound (see Theorem \ref{thm:est-err} in the Appendix for details)
\begin{align}\label{eq:loss-est}
\max_{s\in V}||\delta_{s}||_{1}=\max_{s\in V}||\hat{\theta}_{s}-\theta_{s}^{*}||_{1}
\le 
\frac{16}{\gamma_{G}}s_{0}\lambda
\end{align}
The bound corresponds to the one given in \citet[][Corollary 2, discussed in Section 4.4]{Wainright:2012a}, and the one in \citet[][Lemma 6.8]{Buhlmann:2011}. Because we require the smallest $\gamma$ such that the RE assumption holds, we have that this bound holds simultaneously for all nodes in the Ising graph. 

The bounds on prediction and estimation are important to know the circumstances for the statistical guarantees. However, in many practical situations we cannot be certain of the assumptions of sparsity and restricted eigenvalues. These assumption cannot be checked. And so it becomes relevant to know what the consequences for prediction and estimation are when the assumptions are not satisfied. This is what we investigate next. 

\section{Violation of sparsity and restricted eigenvalues}\label{sec:violations}
If we violate either the sparsity or restricted eigenvalue assumption, then we would expect that lasso estimation error becomes worse, and indeed this happens. However, this is not so clear for prediction. In fact, it turns out that prediction becomes better for non-sparse models that violate the restricted eigenvalue (RE) assumption.
Our main result is that violating the RE or sparsity assumption leads to a decrease in empirical risk, and hence in loss. The RE assumption is violated by an extreme case of multicollinearity, namely where some nodes are copies of other nodes. When such copies are connected we call them connected copies. In connected copies the coefficients are proportional to the original ones, such that we do not arbitrarily change the data generating process. One way to view connected copies is to find multiplicative constants for the edge weights that lead to a network with perfect correlations between nodes. We therefore compare prediction and estimation of different situations where we violate the RE or sparsity assumption based on different data generating processes. Proposition \ref{prop:collinearity} shows that the number of connected copies co-determines the decrease in empirical risk, and hence, violating the RE assumption leads to a decrease in risk. Next, in Corollary \ref{cor:sparsity} we show that violating the sparsity assumption leads to either a decrease or increase of empirical risk depending on whether the set of coefficients in the different subsets of nodes is positive or negative, respectively. We illustrate the theoretical results with some simulations in Section \ref{sec:numerical}.

\subsection{Connected copies}\label{sec:con-copies}
Suppose that for some nodes $s,t\in V$ we have that the observations are identical, that is, $x_{i,s}=x_{i,t}$ for all $i$. Then the coefficients obtained with the lasso using the quadratic approximation to the logistic loss in coordinate descent will be identical, i.e., $\hat{\theta}_{s}=\hat{\theta}_{t}$ \citep[][see also the Appendix for a discussion of the coordinate descent algorithm]{Hastie:2015}. This can be seen from the following considerations. By (\ref{eq:second-deriv-emp}) we have that element $(s,s)$ of the second derivative matrix is
\begin{align*}
\nabla^{2}_{ss}\psi_{n}=\frac{1}{n}\sum_{i=1}^{n}\pi(\mu_{i})\pi(-\mu_{i})x_{i,s}^{2}
\end{align*}
and this is the same for element $(t,t)$ since $x_{s}=x_{t}$.  Similarly for  the $s$th element $\nabla_{s}\psi_{n}$  we obtain 
\begin{align*}
\nabla_{s}\psi_{n} = \frac{1}{n}\sum_{i=1}^{n}(-y_{i}+\pi(\mu_{i}))x_{i,s}
\end{align*}
which equals $\nabla_{t}\psi_{n}$ because $x_{s}=x_{t}$. In the coordinate descent algorithm the updating scheme using the quadratic approximation (see the Appendix for details) is at time $q+1$
\begin{align}\label{eq:coord-descent}
\theta^{q+1}_{j} = \theta^{q}_{j} -
\begin{cases}
(\nabla^{2}_{jj}\psi^{q})^{-1}\nabla_{j}\psi^{q} -\lambda
	&\text{if } (\nabla^{2}_{jj}\psi^{q})^{-1}\nabla_{j}\psi^{q} \hspace{.5em}>\lambda\\
(\nabla^{2}_{jj}\psi^{q})^{-1}\nabla_{j}\psi^{q} +\lambda
	&\text{if } (\nabla^{2}_{jj}\psi^{q})^{-1}\nabla_{j}\psi^{q} \hspace{.5em}<-\lambda\\
0
	&\text{if } |(\nabla^{2}_{jj}\psi^{q})^{-1}\nabla_{j}\psi^{q}| \le\lambda.	
\end{cases}
\end{align}
where $(\nabla^{2}_{jj}\psi^{q})^{-1}$ is element $(j,j)$ of the inverse of the second order derivative matrix $\nabla^{2}\psi^{q}$ for step $q$ in the coordinate descent algortihm. Then we obtain in the coordinate descent algorithm $(\nabla^{2}_{ss}\psi^{q}_{n})^{-1}\nabla_{s}\psi^{q}_{n}$ at each step $q$ for both nodes $s$ and $t$, implying that the coefficients are the same. 
So for each node in the nodewise regressions we obtain a Fisher matrix where column $s$ is the same as column $t$. Now if both $s$ and $t$ are in $S_{0}$, then the smallest eigenvalue of $\nabla^{2}\psi_{n,S_{0}}$ is 0, and hence, the RE assumption is violated. We will use this idea of identical nodes to explain why prediction loss becomes better when we violate the RE assumption.

We call a node $t$ in the subset $L\subset V$ a connected copy of $s\in K=V\backslash L$ if $(s,t)\in E$ and $x_{t}=x_{s}$. This says that two directly connected nodes are identical to each other for all $n$ observations. Note that the coefficient between a connected copy and its original must be positive; if the coefficient were negative, then the connected copy would also have to be the reverse of its original, which cannot be true because the variables are defined to be identical. We know from estimation that if a node is a connected copy then the lasso solution is no longer unique \citep{Hastie:2015}. In fact, if $t$ is a connected copy of $s$, then all solutions with $\alpha\hat{\theta}_{s}$ and $(1-\alpha)\hat{\theta}_{t}$, with $0\le \alpha\le 1$ and $\hat{\theta}_{s}$, $\hat{\theta}_{t}$ are estimates of the parameters of nodes $s$ and $t$ respectively, result in identical empirical risk $R_{n,\psi}$ as when those connected copies have been deleted. Similarly, we will have the same $\ell_{1}$ norm as when the connected copies have been deleted. As a consequence, we cannot distinguish between the situation with or without the connected copy in $\ell_{1}$ optimisation. We denote by $L_{t}$ the set of all connected copies $s\in L_{t}$ of $t\in K$, which defines an equivalence relation on $L$, such that $L_{t}\cap L_{s}=\varnothing$ and $\cup L_{t}=L$. We denote the parameter vector where the connected copies in $L$ have been deleted by $\theta_{\backslash L}$ and correspondingly $\mu_{\backslash L}=x^{\sf T}_{\backslash L}\theta_{\backslash L}$. 
\begin{lemma}\label{lem:copy}
In the Ising graph $G=(V,E)$ suppose nodes in $L\subset V$ are connected copies of nodes in $K=V\backslash L$. Furthermore,  the nodewise lasso solutions $\hat{\theta}$ are obtained with (\ref{eq:lasso}) where for each connected copy $t\in L_{t}$ of node $s\in  K$, with $\alpha_{t}\hat{\theta}_{t}$, we have that $\sum_{t\in L_{t}}\alpha_{t}=1$. Then the empirical risk $R_{n,\psi}(\hat{\mu})$ and $\ell_{1}$ norm of $\hat{\theta}$ are the same as when the connected copies in $L$ are deleted, i.e., $R_{n,\psi}(\hat{\mu})=R_{n,\psi}(\hat{\mu}_{\backslash L})$ and $||\hat{\theta}||_{1}=||\hat{\theta}_{\backslash L}||_{1}$. 
\end{lemma}
So we have that the non-uniqueness of the lasso in case of a connected copy, results in the exact same value for the empirical risk whether we delete it or take any one of the weighted versions such that the coefficients sum to 1. Note that we do not change the underlying process in any arbitrary way; the nodes are connected and the coefficients remain proportional to the original ones. We immediately obtain that the size $|L|$ of the set of connected copies co-determines the prediction loss. We obtain this result because the coefficients of the connected copies with respect to their originals is positive. 
\begin{proposition}\label{prop:collinearity}
For the Ising graph, let $L_{1}$ and $L_{2}$ be subsets of connected copies of nodes in $V\backslash L_{1}\cup L_{2}$ such that $L_{1}\subset L_{2}$ and hence $|L_{1}|< |L_{2}|$. Then we have for the prediction loss that the sum of coefficients in $L_{1}^{c}\cap L_{2}$ is $>0$, and the risk $R_{n,\psi}(\hat{\mu}_{\backslash L_{1}})\ge R_{n,\psi}(\hat{\mu}_{\backslash L_{2}})$.
\end{proposition}
This follows from Lemma \ref{lem:copy} directly, since there we saw that the prediction loss including connected copies is equal to the prediction error when those connected copies are deleted. This idea explains why the empirical risk decreases as a function of an increasing number of connected copies. 

The same idea can be used to determine why prediction becomes better for non-sparse sets. Proposition \ref{prop:collinearity} can be altered such that a similar result holds for sparsity, where we do not need the connected copies. The only requirement is that we know what the sum of the coefficients is that are in the larger set of connected nodes,  because the nodes need not be connected in this case. Let the $S_{a}$ be a set of nodes with a possibly non-sparse set of nonzero edges in the sense that $|S_{a}|>O(\sqrt{n/\log p})$. Suppose that $S_{0}\subset S_{a}$ so that $|S_{0}|<|S_{a}|$. 
\begin{corollary}\label{cor:sparsity}
In the Ising graph $G=(V,E)$ suppose that we have a particular, not necessarily sparse, node set with nonzero edges in $S_{a}$, and define the subset $S_{0}\subset S_{a}$. Then we obtain for the empirical risk $R_{n,\psi}$ that
\begin{itemize}
\item[(1)] if the sum of coefficients in $S_{0}^{c}\cap S_{a}$ is $>0$, then $R_{n,\psi}(\hat{\mu}_{\backslash S_{0}})\ge R_{n,\psi}(\hat{\mu}_{\backslash S_{a}})$;
\item[(2)] if the sum of coefficients in $S_{0}^{c}\cap S_{a}$ is $<0$, then $R_{n,\psi}(\hat{\mu}_{\backslash S_{0}})\le R_{n,\psi}(\hat{\mu}_{\backslash S_{a}})$.
\end{itemize}

\end{corollary}
We see that by eliminating the requirement of connectedness, we find that prediction loss decreases given that the coefficients in the remaining set of non-zero coefficients is positive.

We focus here on prediction loss because by (\ref{eq:prediction-bound-estimation}) we have that the $\ell_{1}$ estimation error is larger than prediction loss (given that the penalty parameter $\lambda$ is of the right order), and hence if we find that prediction loss becomes higher, it follows that $\ell_{1}$ estimation error becomes larger.

The above presented ideas of violating the sparsity assumption or restricted eigenvalue assumption are confirmed by some numerical illustrations.

\subsection{Numerical illustration}\label{sec:numerical}
To show the effects of non-sparse underlying representations and violation of  the restricted eigenvalue assumption  (multicollinearity), we performed some simulation studies. Here 0-1 data was generated by a Metropolis-Hastings algorithm, implemented in the {\small\sf R} package {\em IsingSampler} \citep{Borkulo:2014}, according to a random graph (Erd\"{o}s-Renyi) with $p=100$ nodes and $n=50$ observations. All edge coefficients were positive, so that we expect the prediction error to improve with increasing collinearity. Sparsity of the graph was varied by the probability of an edge from $p_{e}=0.025$, which complies with the sparsity assumption, to the probability of an edge of $p_{e}=0.2$, which does not comply with the sparsity assumption. For interpretation we defined sparsity in these simulations as $1-p_{e}$, so that high sparsity means few non-zero edges. Multicollinearity was induced by equating two columns of the data $X$ if there was an edge in the edge set of the true graph for a percentage $\alpha$, ranging from 0 to 0.6. This ensured that the smallest $\alpha s_{0}$ eigenvalues of the submatrix $\nabla^{2}\psi_{n,S_{0}}$ are 0, thereby violating the RE assumption. 

The parameters for the nodes $m$ and for interactions in $A$ were estimated by nodewise logistic regressions, implemented in {\em IsingFit} \citep{Borkulo:2014}. Here the extended Bayesian information criterion (EBIC) is used to determine the optimal $\lambda$ for each logistic regression separately \citep{Foygel:2013}. This procedure was run 100 times and the averages across these runs (and nodes) are presented. We evaluated estimation accuracy by recall ($|\hat{S}\cap S_{0}|/ |S_{0}|$) and precision ($|\hat{S}\cap S_{0}|/ |\hat{S}|$). We also used a scaled $\ell_{1}$ norm for the estimation error $||\delta||_{1}/u$, where $\delta=\hat{\theta}-\theta^{*}$ and $u$ is the maximal value obtained. Prediction was evaluated by  logistic loss $\psi$ and Bayes loss $\mathcal{C}$. We determine loss for data $z_{i}$ independent from data $y_{i}$, upon which the estimate $\hat{\theta}$ is based (predictive risk).

\begin{figure}[t!]
\begin{tabular}{c @{\hspace{2em}} c}
\includegraphics[width=5.8cm]{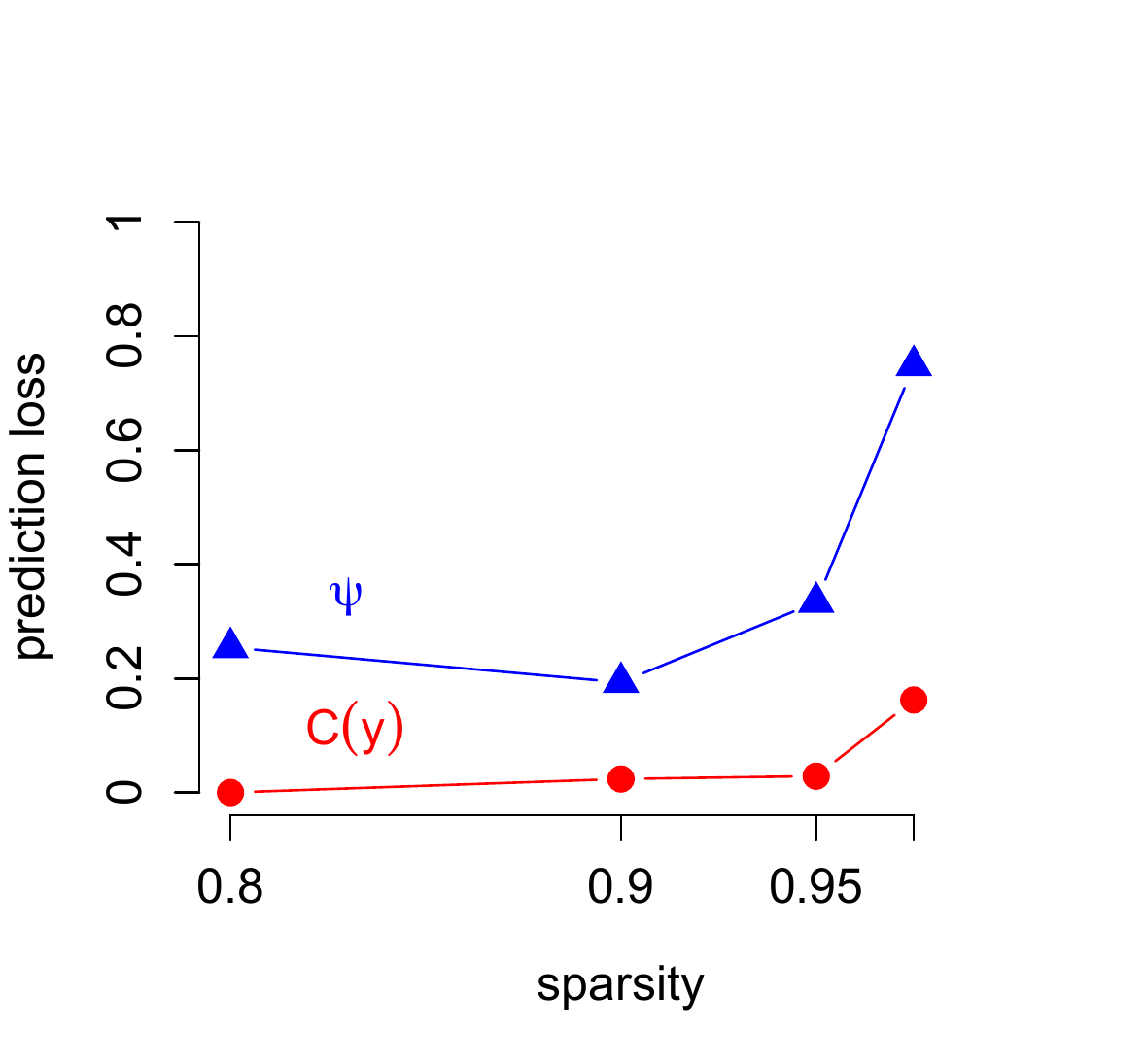}
&
\includegraphics[width=5.8cm]{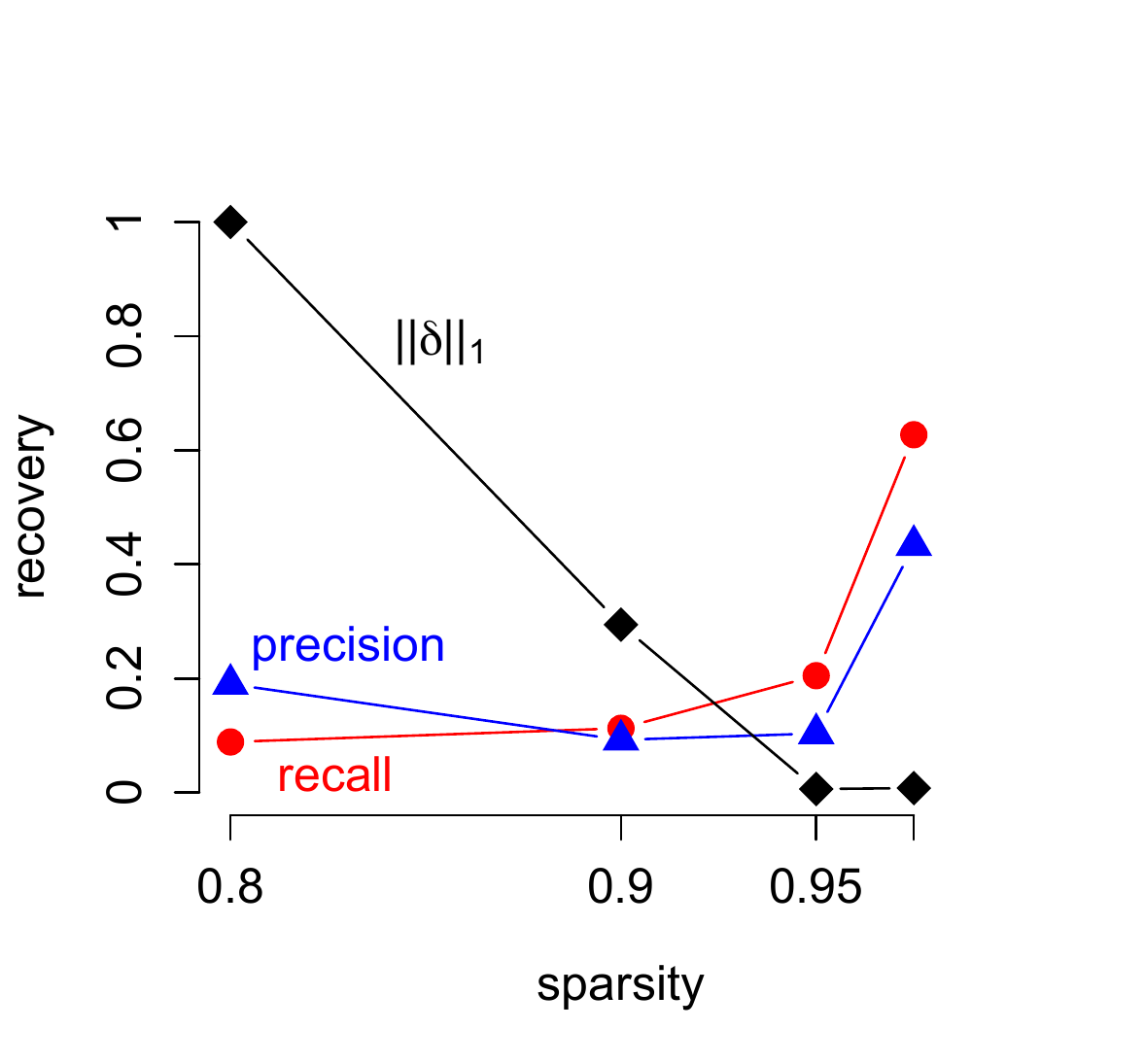}\\
(a)	&(b)
\end{tabular}
\caption{Performance measures of constructing networks with the lasso as a function of sparsity, where sparsity is defined as $1-p_{e}$, the reverse of the edge probability. In (a) Bayes loss (${\color{red}\bullet }$) and logistic loss (${\color{blue} \blacktriangle}$)  and in (b) recovery in terms of recall (${\color{red} \bullet}$) and precision  (${\color{blue} \blacktriangle}$) and the scaled $\ell_{1}$ norm of the error (${\color{black} \blacklozenge}$). }
\label{fig:sparsity}
\end{figure}

Figure \ref{fig:sparsity}(b) shows that recovery of parameters is accurate when sparsity is high (few non-zero edges), but recovery becomes poor when sparsity does not hold; from sparsity 0.95 and lower. This is seen in all three measures: recall, precision and the scaled $\ell_{1}$ norm. In contrast, the 0-1 loss from (\ref{eq:classifier}) and the logistic loss in (\ref{eq:psi-loss}) actually become better (the loss decreases) when the data generating process is no longer sparse, as can be seen in Figure \ref{fig:sparsity}(a). This corresponds to Corollary \ref{cor:sparsity}, which shows that sparsity is not necessary for accurate prediction. We do require that the penalty parameter $\lambda$ is of the appropriate order (i.e., $\lambda=O(\sqrt{\log p/n})$); here $\lambda$ was selected by the EBIC \citep{Foygel:2013} which ensured such a penalty. The EBIC has an additional hyperparameter $\gamma$ to control the impact of the size of the search domain; we set $\gamma$ to $0.25$ in line with the reasonable performance obtained in \citet{Foygel:2013}. Prediction loss is high at high sparsity because in the simulation there are only about 2 to 3 edges, which means that prediction of other nodes is extremely difficult. 

In Figure \ref{fig:collinearity} the results can be seen when multicollinearity is varied.  As expected, Figure \ref{fig:collinearity}(b) shows that increasing multicollinearity reduced recovery; both recall and precision decreased to around 10\%. Prediction loss, on the other hand, becomes smaller as shown in Figure \ref{fig:collinearity}(a), indicating better prediction for multicollinear data. This is in line with Proposition \ref{prop:collinearity}. We can also think of it in the following way. With increasing multicollinearity $\alpha$, more equal columns in $X$ are present for connected nodes. This leads to more similar behaviour of connected nodes in the Ising network and hence to better prediction. 

\begin{figure}[t!]
\begin{tabular}{c @{\hspace{2em}} c}
\includegraphics[width=5.8cm]{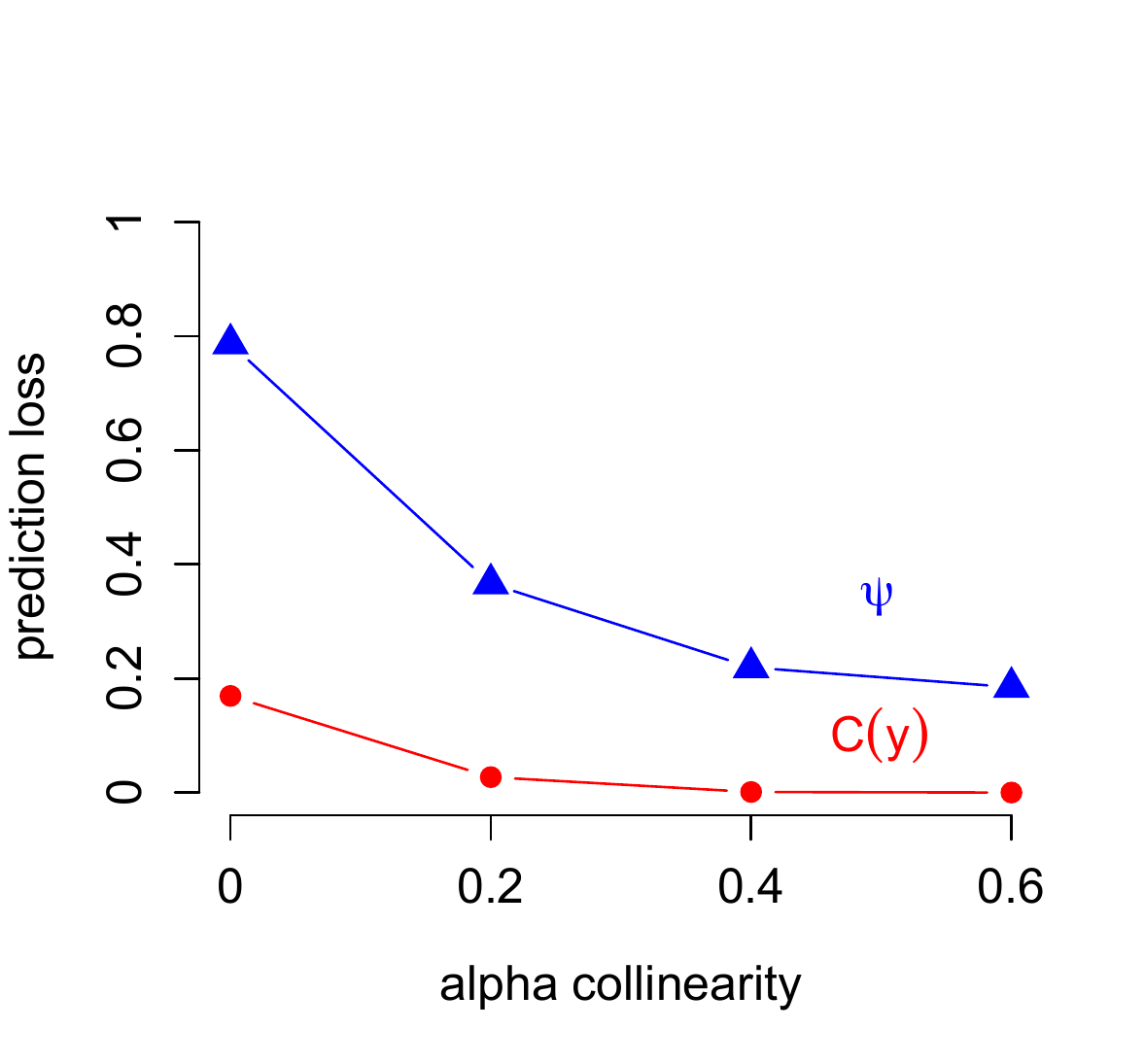}
&
\includegraphics[width=5.8cm]{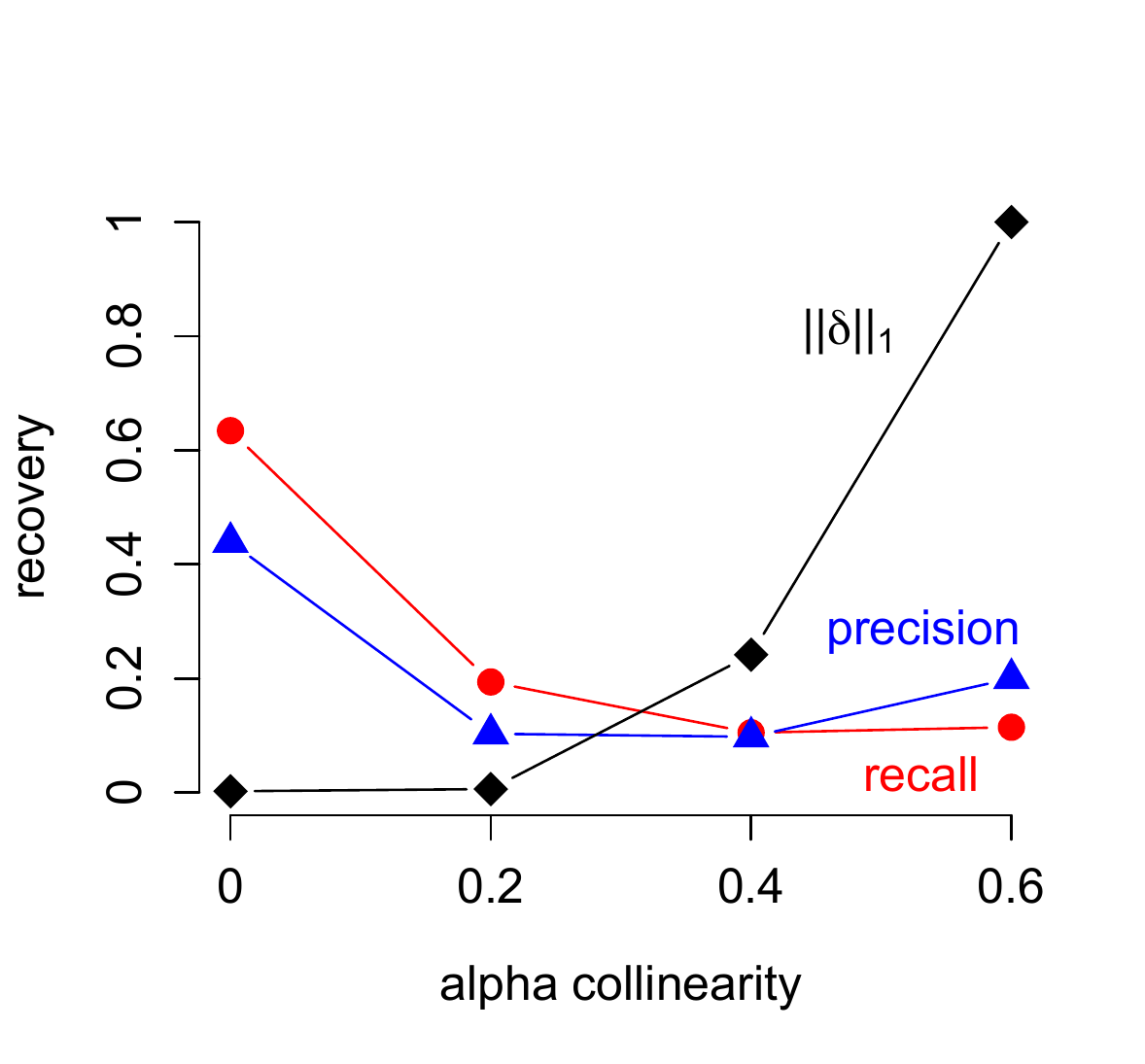}\\
(a)	&(b)
\end{tabular}
\caption{Performance measures of constructing networks with the lasso as a function of collinearity ($\alpha$); collinearity is defined as the probability of identical observations for two nodes whenever these nodes are connected. In (a) Bayes loss (${\color{red}\bullet }$) and logistic loss (${\color{blue} \blacktriangle}$)  and in (b) recovery in terms of recall (${\color{red} \bullet}$) and precision  (${\color{blue} \blacktriangle}$) and the scaled $\ell_{1}$ norm of the error (${\color{black} \blacklozenge}$). }
\label{fig:collinearity}
\end{figure}

These results demonstrate that when either the sparsity assumption or multicollinearity (RE) assumption is violated, the prediction loss decreases, making prediction better. But also that estimation error increases. Hence, the estimated network that predicts well, will not be similar to the true underlying network. On the other hand, if the assumptions of sparsity and RE hold, then many of the edges in the Ising model are estimated correctly but because of the high-dimensional setting many true edges are also missed. And since in sparse settings fewer edges are present that determine the prediction, prediction is poorer.
\section{Discussion}\label{sec:discussion}
Logistic regression is an appropriate tool for prediction and estimation of parameters of the Ising model. Statistical guarantees have been given for prediction and estimation of the parameters of the Ising model using a sequence of logistic regressions whenever at least the assumptions of sparsity and restricted eigenvalues hold. Here we focused on violations of these assumptions and showed why prediction becomes better whenever sparsity or restricted eigenvalues do not hold. Intuitively, for prediction the underlying structure of the graph is irrelevant and when nodes behave similarly, prediction becomes easier. To confirm these intuitions we showed, using connected copies, that prediction loss can decrease as a function of multicollinearity and sparsity. When multicollinearity increases or sparsity decreases, then prediction loss decreases. By consequence of the fact that prediction loss can be considered as a lower bound for estimation error (albeit not a tight bound), estimation error is seen to become worse (increase) as multicollinearity increases or sparsity decreases.   Our simulations support these findings and additionally show that recovery in terms of precision and recall becomes worse when violating the assumption of sparsity and multicollinearity.

The concept of connected copies used here is of course an idealisation of reality. Connected copies can be seen as a way to compare prediction and estimation for different structures (topologies) of graphs, where a connected copy is an extreme case in which the correlation between two variables is 1. We required this idealisation in order to obtain the analytical results. In practice we will not encounter $x_{s}=x_{t}$ but $x_{s}\approx x_{t}$. This case is much more difficult to treat analytically. In the case where $x_{s}\approx x_{t}$ then the parameter estimates will not be equal and the result would depend on the exact differences in estimates. But if we suppose that the sign of all the coefficients is positive, say, then we would expect similar behaviour of the empirical risk based on the results of Proposition \ref{prop:collinearity} and Corollary \ref{cor:sparsity}.

We showed here the consequences of violating the restricted eigenvalue and sparsity assumptions in the Ising model using  logistic regression. The next step is obviously to generalise these results to exponential family distributions. This will require additional restrictions like the margin condition. The margin condition bridges the gap between estimation error and prediction loss. Because for logistic regression we have the linear functional $\mu=\theta^{\sf T}x$, we obtain a quadratic margin. For logistic regression the margin condition then implies that 
$||\hat{\mu}-\mu^{*}||_{2}^{2} \ge \gamma ||\delta ||_{2}^{2}$, where $\delta=\hat{\theta}-\theta^{*}$ and using strong convexity on $\frac{1}{n}\sum_{i=1}^{n}x_{i}x_{i}^{\sf T}$. But the margin condition does not hold in general and so requires additional assumptions \citep[see ][]{Buhlmann:2011} to apply the current analysis of the consequences of violating RE and sparsity on estimation and prediction. 

\section*{Conflict of interest statement}
On behalf of all authors, I declare that none of the authors has a conflict of interest. 

\section*{Appendix}\label{sec:appendix}

{\bf Strong convexity}
The function $\psi$ is strongly convex if for some $\gamma>0$ and all $\theta\in \mathbb{R}^{p}$ it holds that
\begin{align}
\psi(y,\mu) - \psi(y,\mu^{*}) \ge \nabla \psi(y,\mu^{*})^{\sf T}(\theta-\theta^{*}) + \frac{\gamma}{2}||\theta-\theta^{*}||_{2}^{2}
\end{align}
with derivative $\nabla\psi=\partial\psi/\partial \theta$ \citep{Boyd:2004}, which for the logistic function is $(-y+\pi(\mu))x$. Because logistic loss is also $L$-Lipschitz continuous we have that each element of the derivative $\nabla\psi$ is bounded above by $L$, and so 
\begin{align}
\psi(y,\mu) - \psi(y,\mu^{*}) \ge  \frac{\gamma}{2}||\theta-\theta^{*}||_{2}^{2}.
\end{align}
%
%
%
This is equivalent to requiring that the second derivative $\nabla^{2}\psi$ has smallest eigenvalue $\gamma$, because if we assume $\nabla^{2}\psi\ge\gamma I$, where $I$ is the identity matrix, then we have
\begin{align}
\psi(y,\mu) - \psi(y,\mu^{*}) \ge  \frac{\gamma}{2}(\theta-\theta^{*})^{\sf T}\nabla^{2}\psi(\theta-\theta^{*})\ge  \frac{\gamma}{2}||\theta-\theta^{*}||_{2}^{2}.
\end{align}
%

\begin{lemma}\label{lem:emp-proc-bound}
Let $\mathbb{D}_{n,s}(\theta_{s})$ be as defined in (\ref{eq:increment}) for each node $s\in V$, and choose $\lambda_{0}=c\sqrt{t^{2}+\log p/n}$ such that for some $c>0$, 
$\mathbb{D}_{i,s}(\theta_{s})\le c$ for all $i$ and $s\in V$. Then for all $t>0$
we have that 
\begin{align*}
\Prob\left(\max_{s \in V}\left|\mathbb{D}_{n,s}(\theta_{s})\right|\le \lambda_{0}\right) \ge 1-2\exp\left(-nt^{2}\right)
\end{align*}
\end{lemma}
\begin{proof}(Lemma \ref{lem:emp-proc-bound})
In order to obtain a bound on the prediction loss, we need to bound the stochastic part in the empirical risk $R_{n,\psi}(\mu)$.
We have by definition of $\hat{\theta}$ that 
\begin{align}\label{eq:opt}
\frac{1}{n}\sum_{i=1}^{n}\psi(y_{i},\hat{\mu}_{i}) + \lambda||\hat{\theta}||_{1}\le \frac{1}{n}\sum_{i=1}^{n}\psi(y_{i},\mu^{*}_{i}) + \lambda||\theta^{*}||_{1}.
\end{align}
Let 
\begin{align}\label{eq:emp-process}
\mathbb{G}_{n}(\theta)=\frac{1}{n}\sum_{i=1}^{n}\left(\psi(y_{i},\mu_{i})-\E\psi(y_{i},\mu_{i})\right)
\end{align}
be the empirical process of the logistic loss indexed by $\theta=(m_{s},(A_{st},t\in V\backslash\{s\}))$ through $\mu_{\theta}(x)$. We can rewrite the left hand side of the above equation (\ref{eq:opt}), by subtracting and adding the theoretical risk, as 
\begin{align*}
\mathbb{G}_{n}(\hat{\theta}) +\frac{1}{n}\sum_{i=1}^{n}\E\psi(y_{i},\hat{\mu}_{i})+ \lambda||\hat{\theta}||_{1} 
\end{align*}
and similarly for the right hand side of (\ref{eq:opt}) with $\theta^{*}$, where we obtain
$\mathbb{G}_{n}(\theta^{*}) +\frac{1}{n}\sum_{i=1}^{n}\E\psi(y_{i},\mu^{*}_{i})+ \lambda||\theta^{*}||_{1}$.
Then plugging these in (\ref{eq:opt}), we obtain
\begin{align}\label{eq:basic-ineq}
\mathcal{L}_{\psi}(\hat{\mu}) + \lambda||\hat{\theta}||_{1}
\le -\left( \mathbb{G}_{n}(\hat{\theta}) - \mathbb{G}_{n}(\theta^{*}) \right)+ \lambda||\theta^{*}||_{1}
\end{align}
Define the increments of the empirical process by
\begin{align}\label{eq:increment}
\mathbb{D}_{n}(\hat{\theta})=\mathbb{G}_{n}(\hat{\theta}) - \mathbb{G}_{n}(\theta^{*}).
\end{align}
If we can bound $|\mathbb{D}_{n}(\theta)|$ such that the influence of the increment is negligible, then we see from (\ref{eq:basic-ineq}) that we can bound prediction loss by the $\ell_{1}$ norm of $\theta^{*}$. 

We define
$\mathbb{D}_{i,s}(\theta) = (\psi_{i,s} - \psi_{i,s}^{*}) - (\E\psi_{i,s} - \E\psi_{i,s}^{*})$, where $\psi_{i,s}$ is the logistic function for observastion $i$ and node $s\in V$.
Note that $\frac{1}{n}\sum_{i=1}^{n}\mathbb{D}_{i,s}(\theta)=\mathbb{D}_{n,s}(\theta)$ as defined in (\ref{eq:increment}), and $\E(\mathbb{D}_{i,s}(\theta))=0$ for all $i=1,\ldots,n$. 
By assumption there is $c>0$ such that $|\mathbb{D}_{i,s}(\theta)|\le c$ for all $i$ and $s$. For some $\lambda_{0}>0$ by Hoeffding's lemma \citep[e.g.,][]{Bousquet:2004,Venkatesh:2013} we have 
\begin{align*}
\Prob\left( \left|\frac{1}{n}\sum_{i=1}^{n}\mathbb{D}_{i,s}(\theta)\right|> \lambda_{0} \right) 
 	&\le 2\exp\left(  - \frac{2n\lambda_{0}^{2}}{c^{2}} \right) 
\end{align*}
And with $\lambda_{0}=c\sqrt{t^{2}+\log p/n}$ for some $t>0$
\begin{align*}
\Prob\left(\max_{s\in V} \left|\frac{1}{n}\sum_{i=1}^{n}\mathbb{D}_{i,s}(\theta)\right|> \lambda_{0} \right) 
 	&\le 2p\exp\left(  - \frac{2nt^{2}}{2} - \frac{2n\log p}{2n} \right) 
\end{align*}
where in the last inequality we have used the union bound to account for all nodes $p= |V|$.
\end{proof}
%

\begin{lemma}\label{lem:pred-loss-const}
Let $\theta\mapsto \mu_{\theta}(x)$ be the linear function for the Ising model (\ref{eq:log-odds}) and $\hat{\theta}$ is the lasso estimate (\ref{eq:lasso}) obtained with $\lambda\ge 2\lambda_{0}$. Let $\theta^{*}$ be the optimum of the theoretical risk $\theta^{*}=\arg \inf_{\theta}R_{\psi}(\mu_{\theta})$. Then for all nodes in $V$ 
such that  
$|\mathbb{D}_{n}(\theta)|\le\lambda_{0}$ we have with probability $1-2\exp(-nt^{2}/2)$
\begin{align}
\mathcal{L}_{\psi}(\hat{\mu}) =\frac{1}{n}\sum_{i=1}^{n}\E(\psi(y_{i},\hat{\mu}_{i})-\psi(y_{i},\mu_{i}^{*}))\le 2\lambda  ||\theta^{*}||_{1}
\end{align}
If $\lambda$ is of order $O(\sqrt{\log p/n})$ and if $||\theta^{*}||_{1}=o(\sqrt{n/\log p})$, then $\mathcal{L}_{\psi}(\hat{\mu})=o(1)$.
\end{lemma}
\begin{proof} (Lemma \ref{lem:pred-loss-const})
By assumption $|\mathbb{D}_{n}(\hat{\theta})|=|\mathbb{G}_{n}(\hat{\theta}) - \mathbb{G}_{n}(\theta^{*})|\le \lambda_{0}$ with high probability (Lemma \ref{lem:emp-proc-bound}) for all $\theta$. Then if we choose $\lambda\ge 2\lambda_{0}$, we have from (\ref{eq:basic-ineq}) with probability $1-2\exp(-nt^{2}/2)$ that
\begin{align*}
\mathcal{L}_{\psi}(\hat{\mu}) + \lambda||\hat{\theta}||_{1}
\le 2\lambda ||\theta^{*}||_{1} 
\end{align*}
from which the result follows. If $\lambda$ is of order $O(\sqrt{\log p/n})$ and if $||\theta^{*}||_{1}=o(\sqrt{n/\log p})$, then $\mathcal{L}_{\psi}(\hat{\mu})=o(1)$.

\end{proof}
%

\begin{lemma}\label{lem:prediction-error-bound}
Choose $\lambda_{0}$ as in Lemma \ref{lem:emp-proc-bound} and assume the $\mathbb{D}_{i,s}(\theta_{s})$ are uniformly bounded by $c$ as in Lemma \ref{lem:emp-proc-bound}. Then the prediction loss is bounded by the estimation error as
\begin{align*}
\mathcal{L}_{\psi}(\hat{\mu}) 
\le  2\lambda(||\hat{\theta}-\theta^{*}||_{1}) 
\end{align*}
\end{lemma}

\begin{proof}
By (\ref{eq:basic-ineq}) we have with high probability that $\mathcal{L}_{\psi}(\hat{\mu}) + \lambda||\hat{\theta}||_{1}
\le \lambda||\theta^{*}||_{1}$. This implies that $\mathcal{L}_{\psi}(\hat{\mu}) \le \lambda(||\theta^{*}||_{1}- ||\hat{\theta}||_{1})$. Also from (\ref{eq:basic-ineq})  we obtain that $||\theta^{*}||_{1}- ||\hat{\theta}||_{1}\ge 0$; and by the reverse triangle inequality we find that $||\theta^{*}||_{1}- ||\hat{\theta}||_{1}\le ||\theta^{*}-\hat{\theta}||_{1}=||\hat{\theta}-\theta^{*}||_{1}$. This completes the proof.
\end{proof}

\begin{theorem}\label{thm:est-err}
For each node $s$ in the Ising graph $G(V,E)$, let $\psi(y,\mu_{\theta})$ be the logistic loss in (\ref{eq:psi-loss}) with positive definite second derivative matrix $\nabla^{2}\psi(\theta)$, and let $\hat{\theta}$ be the lasso estimate (\ref{eq:lasso}) obtained with $\lambda\ge 2\lambda_{0}=O(\sqrt{\log p/n})$. Suppose that the sparsity assumption \ref{ass:sparsity} holds with $s_{0}=o(\sqrt{n/\log p})$ for all nodes and that the RE assumption \ref{ass:RE} holds with $\gamma_{G}>0$. Then for $\delta_{s}\in\mathbb{C}_{1}$
\begin{align*}
\max_{s\in V}||\delta_{s}||_{1}=\max_{s\in V}||\hat{\theta}_{s}-\theta_{s}^{*}||_{1}
\le 
\frac{16}{\gamma_{G}}s_{0}\lambda
\end{align*}
and $\hat{\theta}$ is consistent for $\theta^{*}$.
\end{theorem}
%

%
\begin{proof} (Theorem \ref{thm:est-err})
Using the $\ell_{1}$ norm, we obtain for each node $s\in V$ the lasso error $\delta=\hat{\theta}-\theta^{*}$
\begin{align*}
||\delta||_{1}=||\delta_{S_{0}}||_{1}+||\delta_{S_{0}^{c}}||_{1}.
\end{align*}
Choosing $\lambda>2\lambda_{0}$ as in Lemma \ref{lem:emp-proc-bound}, we have from (\ref{eq:basic-ineq}) that with probability at least $1-2\exp(-nt^{2})$
\begin{align*}
\mathcal{L}_{\psi}(\hat{\mu}) + \lambda||\hat{\theta}_{S_{0}}||_{1}+\lambda||\hat{\theta}_{S_{0}^{c}}||_{1}\le \lambda||\theta^{*}_{S_{0}}||_{1}+\lambda||\theta^{*}_{S_{0}^{c}}||_{1}.
\end{align*}
Note that $||\theta^{*}_{S_{0}}||_{1}-||\hat{\theta}_{S_{0}}||_{1}\le ||\hat{\theta}_{S_{0}}-\theta^{*}_{S_{0}}||_{1}$ and $\theta^{*}_{S_{0}^{c}}=0$. By rearranging the above equation,  for $\lambda>2\sqrt{\log p/n}$, we have by Lemma \ref{lem:pred-loss-const} that 
\begin{align}\label{eq:restrictedS}
 ||\delta_{S_{0}^{c}}||_{1} \le  ||\delta_{S_{0}}||_{1}
\end{align}
with high probability. We therefore have that $\delta\in\mathbb{C}_{1}$. We can then bound $\ell_{1}$ estimation error and obtain
\begin{align}
||\delta||_{1}\le 2||\delta_{S_{0}}||_{1}\le 2\sqrt{s_{0}}||\delta_{S_{0}}||_{2}
\end{align}
where we used the inequality $||\nu||_{1}\le \sqrt{k}||\nu||_{2}$ for $\nu \in \mathbb{R}^{k}$. We can connect the above to prediction loss by considering strong convexity for the restricted setting for $\delta\in \mathbb{C}_{1}$. We have by the RE assumption that for $\gamma_{G}>0$ for all nodes in the graph $G$ that
\begin{align*}
\frac{1}{n}\sum_{i=1}^{n}\psi(y_{i},\hat{\mu}_{i,S_{0}})-\psi(y_{i},\mu^{*}_{i,S_{0}}) \ge \frac{\gamma_{G}}{2}||\delta_{S_{0}}||_{2}^{2}
\end{align*}	
where $\mu_{i,S_{0}}=x_{i,S_{0}}^{\sf T}\theta_{S_{0}}$. 
Using the empirical process $\mathbb{G}_{n}(\theta)$ defined in (\ref{eq:emp-process}) and the increments $\mathbb{D}_{n}(\theta)$ defined in (\ref{eq:increment}), we obtain 
\begin{align*}
\frac{1}{n}\sum_{i=1}^{n}\psi(y_{i},\hat{\mu}_{i,S_{0}})-\psi(y_{i},\mu^{*}_{i,S_{0}}) =\mathbb{D}_{n}(\hat{\theta}_{S_{0}})+\mathcal{L}_{\psi}(\hat{\mu}_{S_{0}}).
\end{align*}	
If $\lambda>2\lambda_{0}$, then by Lemma \ref{lem:emp-proc-bound} the increments $\mathbb{D}_{n}(\theta)<c$ with probability $1-2\exp(-nt^{2})$, and so 
\begin{align*}
2\lambda||\delta||_{1}\ge\mathcal{L}_{\psi}(\hat{\mu})\ge\mathcal{L}_{\psi}(\hat{\mu}_{S_{0}})\ge \frac{\gamma_{G}}{2}||\delta_{S_{0}}||_{2}^{2}\ge \frac{\gamma_{G}}{2s_{0}}||\delta_{S_{0}}||_{1}^{2}
\end{align*}	
where we used the fact that for each subset $S$ of nodes the prediction loss $\mathcal{L}_{\psi}(\mu_{S})\ge 0$. Rearranging and using that $\delta\in \mathbb{C}_{1}$, gives
\begin{align*}
\frac{16}{\gamma_{G}}\lambda s_{0}||\delta||_{1}\ge4||\delta_{S_{0}}||_{1}^{2}\ge ||\delta||_{1}^{2}
\end{align*}	
as claimed.
\end{proof}
%

\noindent
{\bf Coordinate descent}
An optimisation algorithm for the convex lasso problem (\ref{eq:lasso}) can be characterised by the Karush-Kuhn-Tucker (KKT) conditions \citep[see e.g., ][]{boydVandenberghe04}. The function $\theta_{j}\mapsto |\theta_{j}|$ is not differentiable at 0, and so we require a subdifferential $\partial |\theta_{j}|$ for each $j$. The KKT condition is the subgradient
\begin{align}\label{eq:kkt}
&\frac{1}{n}\sum_{i=1}^{n}(-y_{i}+\pi(\hat{\mu}))x_{i} = \lambda \partial||\hat{\theta}||_{1}
\end{align}
where $\hat{\mu}=\mu_{\hat{\theta}}(x_{i})$, and the subdifferential vector $\partial||\hat{\theta}||_{1}$ is
\begin{align*}
\partial|\hat{\theta}_{j}|\in
\begin{cases}
\{-1\}	&\text{if } \hat{\theta}_{j}<0\\
\{1\}	&\text{if } \hat{\theta}_{j}>0\\
[-1,1]	&\text{if } \hat{\theta}_{j}=0\\
\end{cases}
\end{align*}
The left hand side of the KKT condition (\ref{eq:kkt}) is the first derivative of the empirical risk function $R_{n,\psi}$. The KKT condition with the subdifferential makes clear that a solution $\hat{\theta}_{j}$ for all $j$ will be shrunken towards 0 by the penally $\lambda$. If $
|\frac{1}{n}\sum_{i=1}^{n}(-y_{i}+\pi(\hat{\mu}))x_{i,j}| \le \lambda$ then it will be set to 0, otherwise it will be shrunken towards 0 by $\lambda$.

A solution $\hat{\theta}$ that satisfies the KKT condition can be obtained by, for instance, a subgradient or coordinate descent algorithm \citep{Hastie:2015}. The advantage of a convex program is that any local optimum is in fact a global optimum \citep[see e.g., ][]{Bertsimas:1997}. 
In a coordinate descent algorithm each $\theta_{j}$ is obtained by minimisation using the KKT condition and a quadratic approximation to $\psi$ and updated in turn for all $p$ parameters. The parameters $\theta_{j}$ can be updated in turn because the empirical risk $R_{n,\psi}$ is twice differentiable and convex and the $\ell_{1}$ penalty is a sum of convex functions and hence convex \citep[see, e.g., ][chapter 5]{Hastie:2015}. For logistic regression optimisation remains slow since logistic loss needs to be optimised iteratively. Optimisation can be speeded up by using a quadratic approximation to logistic loss to obtain for each coordinate of $\theta$ separately an update. Let $\theta^{t}$ be the update obtained at step $t$ and let $\delta^{t}=\theta-\theta^{t}$. 
Then we obtain for the quadratic approximation the KKT condition
\begin{align}
\frac{1}{n}\sum_{i=1}^{n}\nabla_{j}\psi(y_{i},x_{i}^{\sf T}\theta^{t})+\nabla^{2}_{jj}\psi(y_{i},x_{i}^{\sf T}\theta^{t})\delta^{t}_{j} =\lambda\partial|\delta^{t}_{j}|
\end{align}
which leads to an estimate for each coordinate $j$ in $1,\ldots, p$. We define 
\begin{align}
(\nabla^{2}_{jj}\psi^{t})^{-1}\nabla_{j}\psi^{t}=\frac{1}{n}\sum_{i=1}^{n}(\nabla^{2}_{jj}\psi(y_{i},x_{i}^{\sf T}\theta^{t}))^{-1}\nabla_{j}\psi(y_{i},x_{i}^{\sf T}\theta^{t}).
\end{align}
At time $t+1$ this gives the update
\begin{align}\label{eq:coord-descent}
\theta^{t+1}_{j} = \theta^{t}_{j} -
\begin{cases}
(\nabla^{2}_{jj}\psi^{t})^{-1}\nabla_{j}\psi^{t} -\lambda
	&\text{if } (\nabla^{2}_{jj}\psi^{t})^{-1}\nabla_{j}\psi^{t} \hspace{.5em}>\lambda\\
(\nabla^{2}_{jj}\psi^{t})^{-1}\nabla_{j}\psi^{t} +\lambda
	&\text{if } (\nabla^{2}_{jj}\psi^{t})^{-1}\nabla_{j}\psi^{t} \hspace{.5em}<-\lambda\\
0
	&\text{if } |(\nabla^{2}_{jj}\psi^{t})^{-1}\nabla_{j}\psi^{t}| \le\lambda.	
\end{cases}
\end{align}
This last equation clearly shows how the threshold of the lasso works: If the update is within $\lambda$ of 0, then it is put to 0 exactly, otherwise it is shrunk to 0 by $\lambda$. Nodewise optimisation for the Ising graph is implemented in the {\sf\small R} package {\em IsingFit}  \citep{Borkulo:2014} using {\em glmnet} \citep{Friedman:2010} which employs a coordinate descent algorithm.

\begin{proof}(Lemma \ref{lem:copy})
Suppose $t\in L_{t}=\{t\}$ is a connected copy of $s\in K$. Recall that logistic loss $\psi(y_{i},\mu_{i})=-y_{i}\mu_{i}+\log(1+\exp(\mu_{i}))$ depends on the parameter $\theta$ only in $\mu_{i}=x_{i}^{\sf T}\theta$. Because $x_{s}=x_{t}$ we have the estimate $\hat{\theta}_{s}=\hat{\theta}_{t}$. Then a solution $\hat{\theta}$ obtained with the lasso in (\ref{eq:lasso}) with $\alpha\hat{\theta}_{s}$ and $(1-\alpha)\hat{\theta}_{s}$ yields for all $i$ 
\begin{align*}
\mu_{i}=\sum_{j\in V\backslash \{s,t\}}x_{i,j}\hat{\theta}_{j} + x_{i,s}\alpha\hat{\theta}_{s} + x_{i,t}(1-\alpha)\hat{\theta}_{t}=
\sum_{j\in V\backslash \{t\}}x_{i,j}\hat{\theta}_{j} 
\end{align*}
which equals the version where node $t$ is deleted from the data.
Hence, for any $0\le \alpha\le 1$ the parameter in logistic loss is the same for all $i$ for such connected copies. It follows that the empirical risk $R_{n,\psi}$ is the same for all such copies.
Similarly, the $\ell_{1}$ norm gives 
\begin{align*}
||\hat{\theta}||_{1}=\sum_{j\in V\backslash \{s,t\}}|\hat{\theta}_{j}| + \alpha|\hat{\theta}_{s}| + (1-\alpha)|\hat{\theta}_{t}|=
\sum_{j\in V\backslash \{t\}}|\hat{\theta}_{j}|
\end{align*}
as claimed. If there are multiple connected copies in $L_{t}$, then choose $0\le \alpha_{j}\le 1$ such that $\sum_{j=1}^{|L_{t}|}\alpha_{j}=1$. Then $\mu_{i}$ is again a sum over $V\backslash L_{t}$ for any $t$, implying the same value for $\psi_{i}$ and hence for $R_{n,\psi}$ as when the nodes in $L_{t}$ were deleted. The same holds for the $\ell_{1}$ norm.
\end{proof}
\begin{proof}(Proposition \ref{prop:collinearity}) 
By Lemma \ref{lem:copy} we have that 
\begin{align*}
\mu_{i,\backslash L_{2}}=\mu_{i,\backslash L_{1}}+\sum_{j\in L_{1}^{c}\cap L_{2}} x_{i,j}\theta_{j}
\end{align*}
since $L_{1}\subset L_{2}$ by assumption. It is easily seen that 
\begin{align*}
\psi(y_{i},\mu_{i,\backslash L_{1}}) - \psi(y_{i},\mu_{i,\backslash L_{2}}) = 
	y_{i}(\mu_{i,\backslash L_{2}} -\mu_{i,\backslash L_{1}}) + \log \left( \frac{1+\exp(\mu_{i,\backslash L_{1}})}{1+\exp(\mu_{i,\backslash L_{2}})} \right)
\end{align*}
Recall that $\log(1+\exp(a))\ge a$. Because $y_{i}$ and $x_{i,j}$ are either 0 or 1, we obtain by the assumption of connected copies that $\sum_{j\in L_{1}^{c}\cap L_{2}} x_{i,j}\theta_{j}>0$, and so $\psi(y_{i},\mu_{i,\backslash L_{1}}) - \psi(y_{i},\mu_{i,\backslash L_{2}}) \ge 0$.
\end{proof}
\begin{proof}(Corollary \ref{cor:sparsity}) 
From Proposition \ref{prop:collinearity} we have the first half that shows that if $\sum_{j\in S_{0}^{c}\cap S_{a}} x_{i,j}\theta_{j}>0$, then $\psi(y_{i},\mu_{i,\backslash S_{0}}) - \psi(y_{i},\mu_{i,\backslash S_{a}}) \ge 0$. The other way around is similar:  if $\sum_{j\in S_{0}^{c}\cap S_{a}} x_{i,j}\theta_{j}<0$, then $\psi(y_{i},\mu_{i,\backslash S_{0}}) - \psi(y_{i},\mu_{i,\backslash S_{a}}) \le 0$, which completes the proof.
\end{proof}

\begin{thebibliography}{45}
\expandafter\ifx\csname natexlab\endcsname\relax\def\natexlab#1{#1}\fi
\expandafter\ifx\csname url\endcsname\relax
  \def\url#1{\texttt{#1}}\fi
\expandafter\ifx\csname urlprefix\endcsname\relax\def\urlprefix{URL }\fi

\bibitem[{Bartlett et~al.(2003)Bartlett, Jordan, and McAuliffe}]{Bartlett:2003}
Bartlett, P.~L., Jordan, M.~I., McAuliffe, J.~D., 2003. Large margin
  classifiers: Convex loss, low noise, and convergence rates. In: NIPS.

\bibitem[{Baxter(2007)}]{Baxter:2007}
Baxter, R.~J., 2007. Exactly solved models in statistical mechanics. Courier
  Corporation.

\bibitem[{Bertsimas and Tsitsiklis(1997)}]{Bertsimas:1997}
Bertsimas, D., Tsitsiklis, J., 1997. Introduction to linear optimization.
  Athena Scientific and Dynamic Ideas, Belmont, Massachusetts.

\bibitem[{Besag(1974)}]{Besag:1974}
Besag, J., 1974. Spatial interaction and the statistical analysis of lattice
  systems. Journal of the Royal Statistical Society. Series B (Methodological)
  36~(2), 192--236.

\bibitem[{Bickel et~al.(2009)Bickel, Ritov, and Tsybakov}]{Bickel:2009}
Bickel, P.~J., Ritov, Y., Tsybakov, A.~B., 2009. Simultaneous analysis of lasso
  and dantzig selector. The Annals of Statistics , 1705--1732.

\bibitem[{Borsboom et~al.(2011)Borsboom, Cramer, Schmittmann, Epskamp, and
  Waldorp}]{Borsboom:2011}
Borsboom, D., Cramer, A. O.~J., Schmittmann, V.~D., Epskamp, S., Waldorp,
  L.~J., 11 2011. The small world of psychopathology. PLoS ONE 6~(11), e27407.

\bibitem[{Bousquet et~al.(2004)Bousquet, Boucheron, and Lugosi}]{Bousquet:2004}
Bousquet, O., Boucheron, S., Lugosi, G., 2004. Introduction to statistical
  learning theory. In: Advanced lectures on machine learning. Springer, pp.
  169--207.

\bibitem[{Boyd and Vandenberghe(2004{\natexlab{a}})}]{Boyd:2004}
Boyd, S., Vandenberghe, L., 2004{\natexlab{a}}. Convex optimization. Cambridge
  University Press.

\bibitem[{Boyd and Vandenberghe(2004{\natexlab{b}})}]{boydVandenberghe04}
Boyd, S., Vandenberghe, L., 2004{\natexlab{b}}. Convex optimization. Cambridge
  University Press.

\bibitem[{Brown(1986)}]{Brown:1986}
Brown, L., 1986. Fundamentals of statistical exponential families. Institute of
  Mathematical Statistics.

\bibitem[{B{\"u}hlmann and van~de Geer(2011)}]{Buhlmann:2011}
B{\"u}hlmann, P., van~de Geer, S., 2011. Statistics for High-Dimensional Data:
  Methods, Theory and Applications. Springer.

\bibitem[{B{\"u}hlmann et~al.(2013)}]{Buhlmann:2013}
B{\"u}hlmann, P., et~al., 2013. Statistical significance in high-dimensional
  linear models. Bernoulli 19~(4), 1212--1242.

\bibitem[{Cantor et~al.(2010)Cantor, Lange, and Sinsheimer}]{Cantor:2010}
Cantor, R.~M., Lange, K., Sinsheimer, J.~S., 2010. Prioritizing gwas results: a
  review of statistical methods and recommendations for their application. The
  American Journal of Human Genetics 86~(1), 6--22.

\bibitem[{Cipra(1987)}]{Cipra:1987}
Cipra, B., 1987. An introduction to the ising model. The American Mathematical
  Monthly 94~(10), 937--959.

\bibitem[{Cressie(1993)}]{Cressie:1993}
Cressie, N., 1993. Statistics for spatial data. John Wiley \& Sons.

\bibitem[{Csardi and Nepusz(2006)}]{Csardi:2006}
Csardi, G., Nepusz, T., 2006. The igraph software package for complex network
  research. InterJournal Complex Systems 1695~(5), 1--9.
\newline\urlprefix\url{http://igraph.org}

\bibitem[{Demidenko(2004)}]{Demidenko:2004}
Demidenko, E., 2004. Mixed models: Theory and applications. John Wiley and
  Sons.

\bibitem[{Foygel and Drton(2013)}]{Foygel:2013}
Foygel, R., Drton, M., 2013. Bayesian model choice and information criteria in
  sparse generalized linear models. Tech. rep., University of Chicago.

\bibitem[{Friedman et~al.(2010)Friedman, Hastie, and
  Tibshirani}]{Friedman:2010}
Friedman, J., Hastie, T., Tibshirani, R., 2010. Regularization paths for
  generalized linear models via coordinate descent. Journal of Statistical
  Software 33~(1), 1--22.

\bibitem[{Giraud(2014)}]{Giraud:2014}
Giraud, C., 2014. Introduction to high-dimensional statistics. Vol. 138. CRC
  Press.

\bibitem[{Hastie et~al.(2001)Hastie, Tibshirani, and Friedman}]{Hastie:2001}
Hastie, T., Tibshirani, R., Friedman, J., 2001. The Elements of Statistical
  Learning. Springer-Verlag, New York.

\bibitem[{Hastie et~al.(2015)Hastie, Tibshirani, and Wainwright}]{Hastie:2015}
Hastie, T., Tibshirani, R., Wainwright, M., 2015. Statistical learning with
  sparsity: the lasso and generalizations. CRC Press.

\bibitem[{Javanmard and Montanari(2014)}]{Javanmard:2014}
Javanmard, A., Montanari, A., 2014. Confidence intervals and hypothesis testing
  for high-dimensional regression. Tech. rep., arXiv:1306.317.

\bibitem[{Johansen-Berg et~al.(2004)Johansen-Berg, Behrens, Robson, Drobnjak,
  Rushworth, Brady, Smith, Higham, and Matthews}]{Johansen-Berg:2004}
Johansen-Berg, H., Behrens, T. E.~J., Robson, M.~D., Drobnjak, I., Rushworth,
  M. F.~S., Brady, J.~M., Smith, S.~M., Higham, D.~J., Matthews, P.~M., 2004.
  Changes in connectivity profiles define functionally distinct regions in
  human medial frontal cortex. Proceedings of the National Academy of Sciences
  of America 101~(36), 13335--13340.

\bibitem[{Kindermann et~al.(1980)Kindermann, Snell, et~al.}]{Kindermann:1980}
Kindermann, R., Snell, J.~L., et~al., 1980. Markov random fields and their
  applications. Vol.~1. American Mathematical Society Providence, RI.

\bibitem[{Kolaczyk(2009)}]{Kolaczyk:2009}
Kolaczyk, E.~D., 2009. Statistical analysis of network data: Methods and
  models. Springer, New York, NY.

\bibitem[{Lockhart et~al.(2014)Lockhart, Taylor, Tibshirani, Tibshirani,
  et~al.}]{Lockhart:2014}
Lockhart, R., Taylor, J., Tibshirani, R.~J., Tibshirani, R., et~al., 2014. A
  significance test for the lasso. The Annals of Statistics 42~(2), 413--468.

\bibitem[{Loh and Wainwright(2012)}]{Loh:2012}
Loh, P.-L., Wainwright, M., 2012. High-dimensional regression with noisy and
  missing data: Provable guarantees with nonconvexity. Annals of Statistics
  40~(3), 1637--1664.

\bibitem[{Marsman et~al.(2017)Marsman, Waldorp, and Maris}]{Marsman:2017}
Marsman, M., Waldorp, L., Maris, G., 2017. A note on large-scale logistic
  prediction: Using an approximate graphical model to deal with collinearity
  and missing data. Behaviormetrika 44~(2), 513--534.

\bibitem[{Meinshausen and B\"{u}hlmann(2006)}]{Meinshausen:2006}
Meinshausen, N., B\"{u}hlmann, P., 2006. High-dimensional graphs and variable
  selection with the lasso. The Annals of Statistics 34~(3), 1436--1462.

\bibitem[{Negahban et~al.(2012)Negahban, Ravikumar, Wainwright, and
  Yu}]{Wainright:2012a}
Negahban, S.~N., Ravikumar, P., Wainwright, M.~J., Yu, B., 2012. A unified
  framework for high-dimensional analysis of m-estimators with decomposable
  regularizers. Statistical Science 27~(4), 538--557.

\bibitem[{Nelder and Wedderburn(1972)}]{Nelder:1972}
Nelder, J.~A., Wedderburn, R. W.~M., 1972. Generalized linear models. Journal
  of the Royal Statistical Society. Series A (General) 135~(3), 370--384.

\bibitem[{P{\"o}tscher and Leeb(2009)}]{Potscher:2009b}
P{\"o}tscher, B.~M., Leeb, H., 2009. On the distribution of penalized maximum
  likelihood estimators: The lasso, scad, and thresholding. Journal of
  Multivariate Analysis 100~(9), 2065--2082.

\bibitem[{Raskutti et~al.(2010)Raskutti, Wainwright, and Yu}]{Raskutti:2010}
Raskutti, G., Wainwright, M.~J., Yu, B., 2010. Restricted eigenvalue properties
  for correlated gaussian designs. The Journal of Machine Learning Research 11,
  2241--2259.

\bibitem[{Ravikumar et~al.(2010)Ravikumar, Wainwright, and
  Lafferty}]{Ravikumar:2010}
Ravikumar, P., Wainwright, M., Lafferty, J., 2010. High-dimensional ising model
  selection using $\ell_1$-regularized logistic regression. The Annals of
  Statistics 38~(3), 1287--1319.

\bibitem[{van Borkulo et~al.(2014)van Borkulo, Borsboom, Epskamp, Blanken,
  Boschloo, Schoevers, and Waldorp}]{Borkulo:2014}
van Borkulo, C.~D., Borsboom, D., Epskamp, S., Blanken, T.~F., Boschloo, L.,
  Schoevers, R.~A., Waldorp, L.~J., 2014. A new method for constructing
  networks from binary data. Scientific reports 4.

\bibitem[{van~de Geer et~al.(2013)van~de Geer, B{\"u}hlmann, and
  Ritov}]{Geer2013}
van~de Geer, S., B{\"u}hlmann, P., Ritov, Y., 2013. On asymptotically optimal
  confidence regions and tests for high-dimensional models. arXiv preprint
  arXiv:1303.0518 .

\bibitem[{Van~de Geer(2008)}]{Geer:2008}
Van~de Geer, S.~A., 2008. High-dimensional generalized linear models and the
  lasso. The Annals of Statistics , 614--645.

\bibitem[{van~de Geer et~al.(2009)van~de Geer, B{\"u}hlmann,
  et~al.}]{Geer:2009}
van~de Geer, S.~A., B{\"u}hlmann, P., et~al., 2009. On the conditions used to
  prove oracle results for the lasso. Electronic Journal of Statistics 3,
  1360--1392.

\bibitem[{Venkatesh(2013)}]{Venkatesh:2013}
Venkatesh, S., 2013. The theory of probability. Cambridge University Press.

\bibitem[{Wainwright(2009)}]{Wainwright:2009}
Wainwright, M.~J., 2009. Sharp thresholds for high-dimensional and noisy
  sparsity recovery using-constrained quadratic programming (lasso).
  Information Theory, IEEE Transactions on 55~(5), 2183--2202.

\bibitem[{Wainwright and Jordan(2008)}]{Wainwright:2008}
Wainwright, M.~J., Jordan, M.~I., 2008. Graphical models, exponential families,
  and variational inference. Foundations and Trends in Machine Learning
  1~(1-2), 1--305.

\bibitem[{Waldorp(2015)}]{Waldorp:2015b}
Waldorp, L., 2015. Testing for graph differences using the desparsified lasso
  in high-dimensional data. (submitted) .

\bibitem[{Young and Smith(2005)}]{Young:2005}
Young, G., Smith, R., 2005. Essentials of statistical inference. Cambridge
  University Press.

\bibitem[{Zhang and Zhang(2014)}]{Zhang:2014}
Zhang, C.-H., Zhang, S.~S., 2014. Confidence intervals for low dimensional
  parameters in high dimensional linear models. Journal of the Royal
  Statistical Society: Series B (Statistical Methodology) 76~(1), 217--242.

\end{thebibliography}
\end{document}